\documentclass{amsart}
\usepackage{amsfonts,amssymb,amscd,amsmath,enumerate,verbatim,calc}
\usepackage{amscd,amssymb,amsopn,amsmath,amsthm,graphics,amsfonts,enumerate,verbatim,calc}
\usepackage[dvips]{graphicx}
\usepackage[colorlinks=true,linkcolor=red,citecolor=blue]{hyperref}
\input xy
\xyoption{all}
\calclayout

\newcommand{\rt}{\rightarrow}
\newcommand{\lrt}{\longrightarrow}

\newcommand{\st}{\stackrel}

\newcommand{\al}{\alpha}
\newcommand{\la}{\lambda}
\newcommand{\La}{\Lambda}

\newcommand{\lan}{\langle}
\newcommand{\ran}{\rangle}

\newcommand{\C}{\mathbb{C} }
\newcommand{\D}{\mathbb{D} }
\newcommand{\K}{\mathbb{K} }
\newcommand{\N}{\mathbb{N} }
\newcommand{\Z}{\mathbb{Z} }

\newcommand{\CA}{\mathcal{A} }
\newcommand{\CC}{\mathcal{C} }
\newcommand{\CCD}{\mathcal{D} }

\newcommand{\CI}{\mathcal{I} }

\newcommand{\CCL}{\mathcal{L} }
\newcommand{\CM}{\mathcal{M} }

\newcommand{\CQ}{\mathcal{Q} }

\newcommand{\CS}{\mathcal{S} }
\newcommand{\CT}{\mathcal{T} }
\newcommand{\CX}{\mathcal{X} }
\newcommand{\CY}{\mathcal{Y} }
\newcommand{\CW}{\mathcal{W}}
\newcommand{\CV}{\mathcal{V}}
\newcommand{\CU}{\mathcal{U}}

\newcommand{\BZ}{\mathbf{Z}}

\newcommand{\FB}{\mathfrak{B}}
\newcommand{\FC}{\mathfrak{C}}
\newcommand{\FT}{\mathfrak{T}}
\newcommand{\BT}{\mathbf{T}}

\newcommand{\BQ}{\mathbf{Q}}

\newcommand{\BE}{\mathbb{E}}

\newcommand{\X}{\mathbf{X}}
\newcommand{\Y}{\mathbf{Y}}
\newcommand{\W}{\mathbf{W}}

\newcommand{\PP}{\mathbf{P}}

\newcommand{\T}{{\bf{T}}}

\newcommand{\SI}{{\lan {I_*} \ran}}

\newcommand{\Mod}{{\rm{Mod\mbox{-}}}}
\newcommand{\gr}{{\rm gr}\mbox{-}}
\newcommand{\FMod}{{\rm{Mod}}}

\newcommand{\mmod}{{\rm{{mod\mbox{-}}}}}

\newcommand{\Inj}{{\rm{Inj}\mbox{-}}}
\newcommand{\Prj}{{\rm{Prj}\mbox{-}}}
\newcommand{\prj}{{\rm{prj}\mbox{-}}}
\newcommand{\GPrj}{{\GP}\mbox{-}}
\newcommand{\Gprj}{{\Gp\mbox{-}}}
\newcommand{\GInj}{{\GI \mbox{-}}}

\newcommand{\Ginj}{{\Gi \mbox{-}}}

\newcommand{\im}{{\rm{Im}}}
\newcommand{\op}{{\rm{op}}}
\newcommand{\inc}{{\rm{inc}}}
\newcommand{\can}{{\rm{can}}}
\newcommand{\sg}{{\rm{sg}}}
\newcommand{\add}{{\rm{add}\mbox{-}}}

\newcommand{\Add}{{\rm{Add}\mbox{-}}}

\newcommand{\ac}{{\rm{ac}}}
\newcommand{\tac}{{\rm{tac}}}

\newcommand{\rc}{{\rm c}}

\newcommand{\bb} {{\rm{b}}}

\newcommand{\GP}{{\mathcal{GP}}}
\newcommand{\Gp}{{\mathcal{G}p}}
\newcommand{\GI}{{\mathcal{GI}}}
\newcommand{\Gi}{{\mathcal{G}i}}

\newcommand{\HT}{{\rm{H}}}

\newcommand{\Coker}{{\rm{Coker}}}
\newcommand{\Ker}{{\rm{Ker}}}

\newcommand{\Rep}{{\rm Rep}}

\newcommand{\Hom}{{\rm{Hom}}}
\newcommand{\Ext}{{\rm{Ext}}}
\newcommand{\End}{{\rm{End}}}

\newtheorem{theorem}{Theorem}[section]
\newtheorem{corollary}[theorem]{Corollary}
\newtheorem{lemma}[theorem]{Lemma}

\newtheorem{proposition}[theorem]{Proposition}

\theoremstyle{definition}

\newtheorem{con}[theorem]{Construction}
\newtheorem{remark}[theorem]{Remark}
\newtheorem{setup}[theorem]{Setup}
\newtheorem{s}[theorem]{}

\theoremstyle{plain}
\newtheorem{stheorem}{Theorem}[subsection]
\newtheorem{scorollary}[stheorem]{Corollary}
\newtheorem{slemma}[stheorem]{Lemma}
\newtheorem{sproposition}[stheorem]{Proposition}

\theoremstyle{definition}
\newtheorem{sdefinition}[stheorem]{Definition}

\newtheorem{sss}[stheorem]{}

\numberwithin{equation}{section}

\begin{document}

\title[Derived equivalences of functor categories]{Derived equivalences of functor categories }

\author[Asadollahi, Hafezi, Vahed]{Javad Asadollahi, Rasool Hafezi and Razieh Vahed}

\address{Department of Mathematics, University of Isfahan, P.O.Box: 81746-73441, Isfahan, Iran and School of Mathematics, Institute for Research in Fundamental Science (IPM), P.O.Box: 19395-5746, Tehran, Iran }
\email{asadollahi@ipm.ir}

\address{School of Mathematics, Institute for Research in Fundamental Science (IPM), P.O.Box: 19395-5746, Tehran, Iran }
\email{hafezi@ipm.ir}
\email{vahed@ipm.ir}

\subjclass[2010]{18E30, 16E35, 16E65, 16P10, 16G10}

\keywords{Functor category; derived category; derived equivalence; recollement.}

\thanks{This research was in part supported by a grant from IPM (No. 93130216)}

\begin{abstract}
Let $\Mod \CS$ denote the category of $\CS$-modules, where $\CS$ is a small category. In the first part of this paper, we provide a version of Rickard's theorem on derived equivalence of rings for $\Mod \CS$. This will have several interesting applications. In the second part, we apply our techniques to get some interesting recollements of derived categories in different levels. We specialize our results to path rings as well as graded rings.
\end{abstract}

\maketitle

\tableofcontents

\section{Introduction}
Tilting theory is initiated from representation theory of finite dimensional algebras, with origins in the work of Bern\v{s}te\u{i}n,  Gel'fand and  Ponomarev \cite{BGP}. It is known that tilting theory can be viewed as a generalization of classical Morita theory; see e.g. \cite{Hap, CPS, Ric, Kel}.

In this direction, one of the most beautiful results is the Rickard's theorem \cite[Theorem 6.4]{Ric} that characterizes all rings that are derived equivalent to a given ring $A$ by determining all tilting complexes over $A$.

On the other hand, functor categories were introduced in representation theory by Auslander \cite{Aus71, Aus74}. He used this kind of categories to classify artin algebras of finite representation type \cite{Aus71} as well as to prove the first Brauer-Thrall conjecture \cite{Aus78}.

Let $\CS$ be a small category. We denote by $\Mod \CS$ the category of all additive contravariant functors from $\CS$ to $\CA b$, the category of abelian groups. $\Mod \CS$, known also as functor category, will be called the category of modules on $\CS$. It is known that it is an abelian category having enough projective object. Our first attempt in this paper is to get a generalization of Rickard's theorem for $\Mod \CS$. To this end, we fix a set of objects $\CT$ of $\K(\Mod \CS)$ satisfying the following properties:
\begin{itemize}
\item[$({\rm P}1)$] $\Hom_{\K(\Mod \CS)}(T , \oplus X_i) \cong \oplus \Hom_{\K(\Mod \CS)}(T , X_i)$, for all $ T, X_i \in \CT$;
\item[$({\rm P}2)$] $\Hom_{\K(\Mod \CS)}(T, T'[i])=0$, for all $i \neq 0$ and all $T, T' \in \CT$;
\item[$({\rm P}3)$] There exists a fixed integer $n$  such that for each $T \in \CT$, $T^i =0$, for $|i|>n$.
\end{itemize}

Theorem \ref{relativeq}, then shows that there is an equivalence of triangulated categories
$$\D(\Mod \CT) \simeq \D_{\CT}(\Mod \CS),$$
where $\D_{\CT}(\Mod \CS)$ is the relative derived category of $\CS$ with respect to $\CT$, defined in \ref{reldercat}. This has several interesting applications. Let us explain some of them in this introduction.

A small category $\CS$ is called $R$-flat, where $R$ is a commutative ring, if $\CS(x,y)$ is flat $R$-module, for every $x, y \in \CS$.
Keller \cite[9.2, Corollary]{Kel} proved that two $R$-flat categories $\CS$ and $\CS'$ are derived equivalent, i.e. $\D(\Mod \CS) \simeq \D(\Mod \CS')$, if and only if there exists a special subcategory $\CT$ of $\K^{\bb}(\prj \CS)$, called tilting subcategory for $\CS$, such that $\CT$ is equivalent to $\CS'$.

Using Theorem \ref{relativeq} we provide a sufficient condition for derived equivalence  of functor categories without flatness assumption on the categories involved. In fact, we show that if $\CT$ is as above and moreover we know that a complex $\X$ in $\K(\Mod \CS)$ is acyclic if and only if it is $\CT$-acyclic, then there exists the following equivalence of triangulated categories
\[\D(\Mod \CT)  \simeq  \D(\Mod \CS).\]
A complex $\X$ in $\K(\Mod \CS)$ is called $\CT$-acyclic, if for every object $T \in \CT$ the induced complex $\Hom_{\CT}(T, \X)$ is acyclic.

Neeman \cite{N08} proved that if $A$ is a left coherent ring, then $\K(\Prj A)$, the homotopy category of projective (right) $A$-modules, is compactly generated and is an infinite completion of $\D^{\bb}(\mmod A^{\op})^{\op}$. This in turn implies that if $A$ and $B$ are two right and left coherent rings such that $\K(\Prj A) \simeq \K(\Prj B)$, then $A$ and $B$ are derived equivalent, i.e. $\D^{\bb}(\Mod A) \simeq \D^{\bb}(\Mod B)$. As another application of our results, we prove the converse of this fact, without any assumption on $A$ and $B$, see Theorem \ref{Thmder} below. As a corollary of this equivalence, we get that if $A$ and $B$ are derived equivalent virtually Gorenstein algebras, then $\underline{\Gprj}A \simeq \underline{\Gprj}B$, where $\underline{\Gprj}A$ is the stable category of finitely generated Gorenstein projective $A$-modules modulo projectives, see preliminaries section for details. This result should be compared with \cite[Theorems 3.8 and 4.2]{Ka} and \cite[Theorem 8.11]{Be}.

The next application we shall explain here is based on Theorem \ref{GorversDer}, that shows that if $A$ and $B$ are two noetherian rings that are ${\rm fGd}$-derived equivalent, then there is an equivalence $\D(\Mod \Gprj A) \simeq \D (\Mod \Gprj B)$ of triangulated categories. For definition of an ${\rm fGd}$-derived equivalence see Definition \ref{fGd}. Such equivalences have been studied by Kato \cite{Ka}.

Recently, infinitely generated tilting modules over arbitrary rings has received considerable attention. In this connection, the notion of good tilting modules is introduced. We apply our results to show in Theorem \ref{inftilt} that every good $n$-tilting module $T_A$ provide an equivalence between the derived category $\D(\Mod A)$ and the relative derived category $\D_{T}( \Mod \End_A(T)^{\op})$. Again we refer the reader to Subsection 3.2 for definition of good $n$-tilting modules and their properties.

Recollements of triangulated categories are `exact sequences' of triangulated categories, which describe the middle term by a triangulated subcategory and a triangulated quotient category. Recollements were introduced by Beilinson, Bernstein and Deligne \cite{BBD} in a geometric context, in order to decompose derived categories of sheaves into two parts, an open and a closed one, and thus providing a natural habitat for Grothendieck's six functors.

A necessary and sufficient condition for the existence of recollements of (bounded) derived module categories of rings has been given by Koenig \cite{Kon}. This result was extended to differential graded rings and unbounded derived categories \cite{J3, NS}. All these criterions characterize the existence of a recollement by determining two exceptional objects.

If the quotient or the subcategory, in an arbitrary recollement, vanishes, then one exceptional object will be vanish  and the other one is a tilting complex, that is, one recovers Morita theory of derived categories.

Section 4 is devoted to investigate the existence of $\D^-(\Mod )$, resp. $\D(\Mod)$, level recollements of functor categories. Let $\CS$ be a small category. We prove that if $\FB$ and $\FC$ are sufficiently nice full triangulated small subcategories  of $\K^{\bb}(\Prj \CS)$, then $\D(\Mod \CS)$ admits a recollement
\vspace{0.2cm}
\[\xymatrix@C=0.5cm{\D(\Mod \FB) \ar[rrr]  &&& \D(\Mod \CS) \ar[rrr] \ar@/^1.5pc/[lll] \ar@/_1.5pc/[lll] &&& \D(\Mod \FC).\ar@/^1.5pc/[lll] \ar@/_1.5pc/[lll] }\]
\vspace{0.1cm}
\\Then it is shown that the above recollement can be restricted to a recollement
\vspace{0.3cm}
\[\xymatrix@=0.5cm{ \D^-(\Mod \FB) \ar[rrr]   &&&  \D^-(\Mod \CS) \ar[rrr]  \ar@/^1.5pc/[lll] \ar@/_1.5pc/[lll] &&&  \D^-(\Mod \FC). \ar@/^1.5pc/[lll] \ar@/_1.5pc/[lll]  }\]
\vspace{0.1cm}

Here, we plan to provide a sufficient condition for the existence of recollements of functor categories. Using this, it will be proved that if there is a $\D^-(\Mod)$ level recollement of rings, then we have a $\D^-(\Mod )$ level recollement of their path rings, incidence rings and monomial rings over any locally finite quiver. This result can be considered as an extension of Asashiba's result \cite{As} that states if $A$ and $B$ are algebras that are derived equivalent, then their path algebras, incidence algebras and monomial algebras are derived equivalent.

It is known that the category of complexes over ring $A$ is equivalent to the category of graded modules over a graded ring $A[x]/(x^2)$, $\gr A[x]/(x^2)$, see \cite{GH}. A similar argument implies that $\Rep (A_{-\infty}^{+\infty}, A)$  is equivalent to $\gr A[x]$. Using these facts, in the last subsection of the paper we specialize our result to get $\D^-(\gr)$ level recollements of graded rings.

\section{Preliminaries}
In general $A$ denotes an associative ring with identity. We let $\Mod A$, resp. $A \mbox{-} {\rm Mod}$, denote the category of all right, resp. left, $A$-modules. We also, consider the following full subcategories of $\Mod A$.
\begin{itemize}
\item[$~$] $\Prj A$ $=$ \ projective (right) $A$-modules,
\item[$~$] $\mmod A$ $=$ \ finitely presented  $A$-modules,
\item[$~$] $ \prj A$ $=$ \ finitely generated projective $A$-modules.
\end{itemize}

\s Let $\CC$ be an additive category. We denote by $\C(\CC)$ the category of complexes in $\CC$. We grade complexes cohomologically, so every complex $\X$ in $\C(\CC)$ is of the form
\[ \cdots \lrt X^{n-1} \st{\partial^{n-1}}\lrt X^n \st{\partial^n}\lrt X^{n+1} \lrt \cdots.\]
Let $\X=(X^i , \partial^i)$ be a complex in $\C(\CC)$ and $n , m$ be integers. We define the following truncations of $\X$:
\[ {}_{\sqsubset_n} \X: \ 0 \lrt X^n \st{\partial^n} \lrt X^{n+1} \st{\partial^{n+1}} \lrt X^{n+2} \lrt \cdots \]
\[ \X {}_{{}_m\sqsupset}: \ \cdots \lrt X^{m-2} \st{\partial^{m-2}} \lrt X^{m-1} \st{\partial^{m-1}} \lrt X^m \lrt 0 \]
\[ {}_{\subset_n} \X : \ 0 \lrt \Coker \partial^{n-1} \st{\bar{\partial^n}} \lrt X^{n+1} \st{\partial^{n+1}} \lrt X^{n+2} \lrt \cdots \]
\[\X {}_{{}_m\supset} : \cdots \lrt X^{m-2} \st{\partial^{m-2}} \lrt X^{m-1} \st{\partial^{m-1}} \lrt \Ker \partial^m \lrt 0 \]
The differential $\bar{\partial^n}$ is the induced map on residue classes.
\vspace{0.2cm}

We denote the homotopy category of $\CC$ by $\K(\CC)$; the objects are complexes and morphisms are the homotopy classes of morphisms of complexes.
The full subcategory of $\K(\CC)$ consisting of all bounded above, resp. bounded, complexes is denoted by $\K^-(\CC)$, resp. $\K^{\bb}(\CC)$. Moreover, we denote by $\K^{ -, \bb}(\CC)$,  the full subcategory of $\K^-(\CA)$ formed by all complexes $\X$ such that there is an integer $n=n(\X)$ with ${\rm H}^i(\X)=0$, for all $i\leq n$.

 \vspace{0.2cm}

 Let $\CA$ be an abelian category. The derived category of $\CA$ will be denoted by $\D(\CA)$. Also, $\D^-(\CA)$, resp. $\D^{\bb}(\CA)$, denotes the full subcategory of $\D(\CA)$ formed by all homologically bounded above, resp. homologically bounded, complexes.

 \s {\sc Total acyclicity.} Let $\CX$ be an additive category. A complex $\X$ in $\C(\CX)$ is called $\CX$-totally acyclic if for every object $Y \in \CX$, the induced complexes $\Hom_{\CX}(\X, Y)$ and $\Hom_{\CX}(Y, \X)$ of abelian groups are acyclic.

Let $\CA$ be an abelian category having enough projective, resp. injective, objects and $\CX= \Prj \CA$, resp. $\CX= \Inj \CA$, be the class of projectives, resp. injectives. In this case, an $\CX$-totally acyclic complex is called totally acyclic complex of projectives, resp. totally acyclic complex of injectives. An object $G$ in $\CA$ is called Gorenstein projective, resp. Gorenstein injective, if $G$ is a syzygy of a totally acyclic complex of projectives, resp. totally acyclic complex of injectives. We denote the class of all Gorenstein projective, resp. Gorenstein injective, objects in $\CA$ by $\GPrj \CA$, resp. $\GInj \CA$. In case $\CA= \Mod A$, we abbreviate the notations to $\GPrj A$ and $\GInj A$. We set $\Gprj A= \GPrj A \cap \mmod A$ and $\Ginj A= \GInj A \cap \mmod A$.
\vspace{0.2cm}

\s {\sc Localizing and thick subcategories.}
Let $\CCD$ be a triangulated category.
A triangulated subcategory $\CCL$ of $\CCD$ is called  thick, if it is closed under direct summands. The smallest full  thick subcategory  of $\CCD$ containing a class $\CCL$ of objects is denoted by ${\rm thick}( \CCL)$. A triangulated subcategory $\CCL$ of $\CCD$ is called localizing if it is closed under all coproducts allowed in $\CCD$. If $\CCL$ is a subclass of objects of $\CCD$, we denote by ${\rm Loc} \CCL$, the smallest full localizing subcategory of $\CCD$ containing $\CCL$.

We have the following constructions.

\begin{con}\label{Cons}(see \cite[Lemma 3.3]{K04}) Given a class  $\CCL$  of objects of a triangulated category $\CCD$, we take $\overline{\CCL}$ to be the class of all $X[i]$ with $X \in \CCL$ and $i \in \Z$. Define a full subcategories $\lan \CCL\ran_n$, for $n>0$, inductively as follows.
\begin{itemize}
\item ${\lan \CCL \ran}_1$ is the  subcategory of $\mathcal{D}$ consisting of all direct summands of objects of $\overline{\CCL}$.
\item For $n>1$, suppose that $\overline{\CCL}_n$ is the class of objects $X$ occuring in a triangle
     $$Y \rt X \rt Z \rightsquigarrow$$   with
$Y \in {\lan \CCL \ran}_{i}$ and $Z\in {\lan \CCL \ran}_j$ such that $i,j <n$.
Let $\lan \CCL\ran_n$ denote  the full subcategory of $\CCD$ formed by all direct summands of objects of $\overline{\CCL}_n$.
\\It can be easily checked that ${\rm thick}( {\CCL}) =\bigcup_{n\in \N}{\lan \CCL \ran}_n$.
\end{itemize}
\end{con}

Let  $\al$ be a regular cardinal. A coproduct in $\CCD$ is said to be $\al$-coproduct, if it is indexed by a set of cardinality less than $\al$. A full thick subcategory of $\CCD$ is called $\al$-localizing if it is  closed under $\al$-coproducts. For a subcategory $\CCL$ of $\CCD$, ${\rm Loc}_{\al}\CCL$ is the smallest $\al$-localizing subcategory of $\CCD$ containing $\CCL$.
 In fact, ${\rm Loc} {\CCL}  =\bigcup_{\al}{\rm Loc}_{\al} \CCL$.
\vspace{0.2cm}

The construction of ${\rm Loc}_{\al}\CCL $ will be used throughout the paper and deserves more attention.
\begin{con}\label{Cons2}
For a class $\CCL$ of objects of $\CCD$, we let $\overline{\CCL}_{\al}$ be the class of all objects $X[i]$, with $X \in \CCL$ and $i \in \Z$, together with their $\al$-coproducts.
\begin{itemize}
  \item  ${\rm Loc}_{\al}^1\CCL$ is the full subcategory of $\CCD$ whose objects are direct summand of objects of $\overline{\CCL}_{\al}$.
  \item For $n>0$, $\overline{\CCL}^n_{\al}$ denotes the class of  all objects ${ X}$ such that there is a triangle $$Y \rt X \rt Z \rightsquigarrow$$
       with
$Y \in {\rm Loc}_{\al}^i\CCL$ and $Z\in {\rm Loc}_{\al}^j\CCL$ such that $i,j <n$.
We define ${\rm Loc}_{\al}^n\CCL$ to be the full subcategory of $\CCD$ formed by all direct summands of $\al$-coproducts of objects of $\overline{\CCL}_{\al}^n$.
\end{itemize}
 Now, it can be easily checked that ${\rm Loc}_{\al}\CCL= \bigcup_{n \in \N} {\rm Loc}_{\al}^n\CCL$.
\end{con}

\s \label{Recoll}{\sc Recollements and stable $t$-structures.} Let $\CT$, $\CT'$ and $\CT''$ be triangulated categories. $\CT$ is called a recollement of $\CT'$ and $\CT''$ if there exists a diagram consisting of six triangulated functors as follows
\[ \xymatrix@C-0.5pc@R-0.5pc{ \CT' \ar[rr]^{i_*=i_!}  && \CT \ar[rr]^{j^*=j^!} \ar@/^1pc/[ll]_{i^!} \ar@/_1pc/[ll]_{i^*} && \CT'' \ar@/^1pc/[ll]_{j_*} \ar@/_1pc/[ll]_{j_!} }\]
satisfying the following conditions:
\begin{itemize}
\item[$(i)$] $(i^*,i_*)$, $(i_!,i^!)$, $(j_!, j^!)$ and $(j^*,j_*)$ are adjoint pairs.
\item[$(ii)$] $i^!j_*=0$, and hence $j^!i_!=0$ and $i^*j_!=0$.
\item[$(iii)$] $i_*$, $j_*$ and $j_!$ are full embeddings.
\item[$(iv)$] for any object $T \in \CT$, there exist the following triangles
$$ i_!i^!(T) \rt  T \rt j_*j^*(T) \rightsquigarrow \ \ \ \text{and} \ \ \ j_!j^!(T) \rt T \rt i_*i^*(T) \rightsquigarrow$$
in $\CT$.
\end{itemize}
\vspace{0.2cm}

Let $\CT_1$ and $\CT_2$ be triangulated categories that admit the following recollements
\[\xymatrix@C=0.5cm{\CT'_1\ar[rrr]^{i_{1*}}  &&& \CT_1 \ar[rrr]^{j^{*}_1} \ar@/^1.5pc/[lll]_{i^{!}_2} \ar@/_1.5pc/[lll]_{i^{*}_1} &&& \CT''_1, \ar@/^1.5pc/[lll]_{j_{1*}} \ar@/_1.5pc/[lll]_{j_{1!}} }\]
\[\xymatrix@C=0.5cm{\CT'_2\ar[rrr]^{i_{2*}}  &&& \CT_2 \ar[rrr]^{j^{*}_2} \ar@/^1.5pc/[lll]_{i^{!}_2} \ar@/_1.5pc/[lll]_{i^{*}_2} &&& \CT''_2. \ar@/^1.5pc/[lll]_{j_{2*}} \ar@/_1.5pc/[lll]_{j_{2!}} }\]
A triangle functor $F: \CT_1 \lrt \CT_2$ is called a morphism of recollements, if there are triangle functors $F': \CT'_1 \lrt \CT'_2$ and $F'': \CT''_1 \lrt \CT''_2$ make the following diagrams commutative, up to natural isomorphism
\[\tiny \xymatrix@C=0.3cm@R=0.4cm{\CT'_1 \ar@{<-_{)}}[dd]^{F'}  &&& \CT_1  \ar@{<-_{)}}[dd]^{F} \ar@/_1.5pc/[lll]_{i^{*}_1} &&& \CT''_1 \ar@{<-_{)}}[dd]^{F''} \ar@/_1.5pc/[lll]_{j_{1!}}
\\ \\  \CT'_2  &&& \CT_2   \ar@/_1.5pc/[lll]_{i^{*}_2} &&& \CT''_2  \ar@/_1.5pc/[lll]_{j_{2!}} }
\ \ \xymatrix@C=0.3cm@R=0.4cm{\CT'_1\ar[rrr]^{i_{1*}} \ar@{<-_{)}}[dd]^{F'}  &&& \CT_1 \ar[rrr]^{j^{*}_1} \ar@{<-_{)}}[dd]^{F}   &&& \CT''_1  \ar@{<-_{)}}[dd]^{F''}
\\ \\  \CT'_2\ar[rrr]^{i_{2*}}  &&& \CT_2 \ar[rrr]^{j^{*}_2} &&& \CT''_2  }
\ \ \xymatrix@C=0.3cm@R=0.4cm{\CT'_1 \ar@{<-_{)}}[dd]^{F'}  &&& \CT_1  \ar@{<-_{)}}[dd]^{F} \ar@/^1.5pc/[lll]_{i^{!}_1}  &&& \CT''_1 \ar@/^1.5pc/[lll]_{j_{1*}} \ar@{<-_{)}}[dd]^{F''}
\\ \\  \CT'_2  &&& \CT_2  \ar@/^1.5pc/[lll]_{i^{!}_2}  &&& \CT''_2. \ar@/^1.5pc/[lll]_{j_{2*}} }
\]

Let $\CCD$ be a triangulated category. A pair $(\CU, \CV)$ of full subcategories of $\CT$ is called an stable $t$-structure, if the following conditions are satisfied
\begin{itemize}
\item[$(i)$] $\mathcal{U}=\Sigma \mathcal{U}$ and $\mathcal{V}= \Sigma \mathcal{V}$,
\item[$(ii)$] $\Hom_{\CT}(\mathcal{U}, \mathcal{V})=0$,
\item[$(iii)$] For every $X \in \CT$, there is a triangle $U \rt X \rt V \rightsquigarrow$ with $U \in \mathcal{U}$ and $V \in \mathcal{V}$.
\end{itemize}

Following proposition provides a connection between the above two notions.
\begin{proposition} \label{Miyachi} \cite{Mi}
Let $(\CU,\CV)$ and $(\CV,\CW)$ be stable $t$-structures in $\CT$. Then there is a recollement
\[\xymatrix@C-0.5pc@R-0.5pc{ \CV  \ar[rr]^{i_*}  && \CT \ar[rr]^{j^*} \ar@/^1pc/[ll]_{i^!} \ar@/_1pc/[ll]_{i^*} && \CT/\CV \ar@/^1pc/[ll]_{j_*} \ar@/_1pc/[ll]_{j_!} }\]
in which  $i_*: \CV \lrt \CT$ is a canonical embedding, $\im j_!=\CU$ and $\im j_* =\CW$.
\end{proposition}

\section{Derived equivalences of functor categories}
In this section we provide a  sufficient condition for derived equivalences of functor categories. Then we present a version of Rickard's Theorem in terms of functor categories. \\

 Let $\CS$ be a skeletally small category. We denote by $\Mod \CS$ the category of contravariant functors from $\CS$ to the category of abelian groups. $\Mod \CS$ is called the category of modules over $\CS$, or the category of (right) $\CC$-modules. It is known that $\Mod\CS$ is an abelian category with arbitrary coproducts.

For an object $S \in \CS$, one may apply Yoneda lemma to show that the representable functor $\Hom_{\CS}(-,S)$ is a projective object in $\Mod\CS$. Moreover, for a functor $F \in \Mod \CS$, there is an epimorphism $\coprod_i \Hom_{\CS}(-, S_i) \lrt F \lrt 0$, where $S_i$ runs through isomorphism classes of objects in $\CS$. So the abelian category $\Mod\CS$ has enough projective objects.  The full subcategory of $\Mod \CS$ consisting of all projective objects will be denoted by $\Prj \CS$.

An object $F$ of $\Mod\CS$ is called finitely presented  if there exists an epimorphism $$\Hom_{\CS}(-, S_1) \lrt \Hom_{\CS}(-, S_0) \lrt F \lrt0,$$ for some objects $S_1$ and $S_0$ of $\CS$. The category of all finitely presented $\CS$-modules is denoted by $\mmod\CS$. Note that a finitely presented $\CS$-module $F$ is projective if and only if it is isomorphic to a direct summand of a finite direct sum of representable functors. We denote by $\prj \CS$ the full subcategory of $\mmod \CS$ consisting of all finitely presented projective functors.

 Let $\CS$ be a small category which is closed under finite direct sums and every idempotent splits. Then, in this case, every finitely presented projective $\CS$-module is of the form $\Hom_{\CS}(-, S)$, for some $S \in \CS$.\\

\s \label{Properties} Let $\CS$ be a small category. Let $\CT$ be a  set of objects in   $\K(\Mod \CS)$ satisfying the following properties:
\begin{itemize}
\item[$({\rm P}1)$] $\Hom_{\K(\Mod \CS)}(T , \oplus X_i) \cong \oplus \Hom_{\K(\Mod \CS)}(T , X_i)$, for all $ T, X_i \in \CT$;
\item[$({\rm P}2)$] $\Hom_{\K(\Mod \CS)}(T, T'[i])=0$, for all $i \neq 0$ and all $T, T' \in \CT$;
\item[$({\rm P}3)$] There exists a fixed integer $n$  such that for each $T \in \CT$, $T^i =0$, for $|i|>n$.
\end{itemize}

Let $\K\K(\Mod \CS):= \K(\K(\Mod \CS))$ denote the homotopy category of $\K(\Mod \CS)$. Thus every object of $\K\K(\Mod \CS)$ is a complex of complexes over $\Mod \CS$ together with a sequence of homotopy classes of maps between them such that composing every two consecutive maps is null homotopic.  By \cite{Ric}, for every object $\X$ in $\K\K(\Mod \CS)$, we have:
\begin{itemize}
\item[$(a)$] a bigraded object $\X^{**}$ of $\Mod \CS$;
\item[$(b)$] a graded endomorphism $d$ of degree $(1,0)$, obtained from the differentials in the complexes over $\Mod \CS$;
\item[$(c)$] a graded endomorphism $\delta$ of degree $(0,1)$, obtained from the original complex.
\end{itemize}

\begin{setup}
Throughout the paper, unless otherwise specified, $\CS$ denotes a small category  and $\CT$ is a set of objects in $\K(\Mod \CS)$ satisfying  the properties defined at \ref{Properties}.
Also, Let $\FT$ denote  a class of objects of $\K\K(\Mod \CS)$ consisting of all shiftings of stalk complexes $\BT$ with $T$ in degree zero, where $T \in \CT$.
\end{setup}
Note that the same property as $({\rm P}1)$ holds true for every complex $\BT$ in $\FT$, i.e.
$$ \Hom_{\K\K(\Mod \CS)}(\BT, \oplus_i \BT_i) \cong \oplus_i \Hom_{\K\K(\Mod \CS)}(\BT,\BT_i),$$
where $\BT_i \in \FT$.

Let $\K\K_{\FT}(\Mod \CS)$ be the smallest full triangulated subcategory of $\K\K(\Mod \CS)$ that contains $\FT$ and is closed under coproducts.

\begin{lemma}\label{Xj}
Let $\X$ be a complex in $\K\K_{\FT}(\Mod \CS)$. Then for every $j$, $(\X^{*j},d)$ is a direct sum of objects of $\FT$.
\end{lemma}
\begin{proof}
By definition, $\K\K_{\FT}(\Mod \CS)= {\rm Loc} \FT = \bigcup_{\al} {\rm Loc}_{\al}\FT$. So, if $\X \in \K\K_{\FT}(\Mod \CS)$, then $\X \in {\rm Loc}_{\al}\FT$ for some regular cardinal $\al$. By Construction \ref{Cons2}, ${\rm Loc}_{\al}\FT = \bigcup_{n \in \N}{\rm Loc}^n_{\al}\FT$ and so we prove the statement by induction on $n$. It is clearly true for $n=1$. Assume that $n>1$. By construction, there is a triangle
$$\Y \rt \X \rt \BZ \rightsquigarrow,$$
where $\Y \in {\rm Loc}^l_{\al}\FT$ and $\BZ \in  {\rm Loc}^m_{\al}\FT$, where $m,l<n$.
Induction hypothesis implies that for every $j$, $(\Y^{*j},d)$ and $(\BZ^{*j}, d)$ are direct sums of objects of $\FT$. So, by the above triangle, $(\X^{*j},d)$ is a direct sum of objects of $\FT$, for every $j \in \Z$.
\end{proof}

\s \label{G(S)} In the following, our aim is to construct a functor from $\K\K_{\FT}(\Mod \CS)$ to $\K(\Mod \CS)$. A complex $\X=(X^{**}, d, \partial) \in \K\K_{\FT}(\Mod \CS)$ is of the following form:
\begin{itemize}
\item[$(i)$] $\X^{ij}$ is zero for large $i$ or large $j$;
\item[$(ii)$] $d^2=0$;
\item[$(iii)$] for every $j$, $(\X^{*j},d)$, considered as a single complexes, is a direct sum of  objects of $\CT$;
\item[$(iv)$] $d\delta =\delta d$;
\item[$(v)$] for every $j$, $\delta^2: \X^{*j} \lrt \X^{*(j+2)}$, considered as a chain map of complexes, is homotopic to zero.
\end{itemize}

Hence, as it is mentioned in \cite[p. 439]{Ric}, if we want to follow the natural way, i.e. consider $\X$ as a double complex over $\Mod \CS$ and then form its total complex, a technical problem will arise. In fact, we shall required $\partial^2$ to be zero, but our assumption only implies that $\partial^2$ is homotopic to zero.  In order to fix the problem, the differentials $d_0=d, d_1=(-1)^{i+j}\delta$ will be extended to a sequence $\{d_i: i=0,1, \cdots\}$ of maps of degree $(1-i,i)$ such that for every $n$,
$$d_0d_n + d_1d_{n-1} + \cdots + d_nd_0=0.$$
Now, one can deduce that the total complex of $\X^{**}$ with $\Sigma d_i$, as differentials, is a complex over $\Mod \CS$.

\begin{lemma}\label{lem1}
Let $\X^{**}$ and $\Y^{**}$ be graded objects of $\Mod \CS$ that are equipped with a graded endomorphism $d$ of degree $(1,0)$ with $d^2=0$. Suppose also that for every $j$, $(\X^{*j},d)$ and $(\Y^{*j},d)$, regarded as single complexes are isomorphic to direct sums of objects of $\FT$. Then for every graded map $\al: \X^{**}\lrt \Y^{**}$ of degree $(p,q)$, $p \neq 0$ and $d\al =\al d$, there is a graded map $h$ of degree $(p-1,q)$ such that $\al=dh+hd$.
\end{lemma}

\begin{proof}
This can be obtained by  straightforward  modification of the argument of the proof of Lemma 2.3 of \cite{Ric}.
\end{proof}

\s\label{3.4} Let $G(\CS)$ denote the category that is defined as follows. An  object is a  system $\{ \X^{**}, d_i \mid i=0, 1, \cdots\}$, where $\X^{**}$ is a bigraded object of $\Mod \CS$ satisfying the conditions as an object in $\K\K_{\FT}(\Mod \CS)$ and, for every $i$,  $d_i$ is a graded endomorphism of $\X^{**}$ of degree $(1-i,i)$ such that for every $n$,
$$d_0d_n + d_1d_{n-1} + \cdots + d_nd_0=0.$$
The morphisms between two systems $\{\X^{**},d_i\}$ and $\{\Y^{**},d_i\}$ are collections $\{\al_i \mid i=0,1, \cdots \}$ of graded maps $\X^{**} \lrt \Y^{**}$ of degrees $(-i,i)$ satisfying
$$\al_0 d_n + \al_1d_{n-1}+ \cdots \al_nd_0=d_0\al_n+ d_1\al_{n-1} + \cdots +d_n\al_0,$$
for each $n$.
\vspace{0.2cm}

\begin{proposition}\label{functor}
There is a functor
$$\Phi: \K \K_{\FT}(\Mod \CS) \lrt \K(\Mod \CS)$$
of triangulated categories that preserves coproducts.
\end{proposition}
\begin{proof}
Let $\X$ be an object of $\K\K_{\FT}(\Mod \CS)$. There is already the bigraded object $X^{**}$ of $\K(\Mod \CS)$. Hence to construct the desired functor, what is required to be defined is the graded endomorphism $d_i$. We have to choose $d_0=d$ and $d_1=(-1)^{i+j}\delta : X^{i,j} \lrt X^{i,j+1}$. Then other $d_i$'s are defined inductively by applying Lemma \ref{lem1}. In a similar way, for any map $\al: \X \lrt \Y$ in $\K\K_{\FT}(\Mod \CS)$, we can assign a morphism $\{\al_i\}:X^{**} \lrt Y^{**}$ of $G(\CS)$. For more details, see \cite[Propositions 2.6 and 2.7]{Ric}.

Altogether, we can now define a triangulated functor
$$\Phi: \K\K_{\FT}(\Mod \CS) \lrt \K(\Mod \CS)$$
by going first to $G(\CS)$ and then taking total complex.
The method that is used in \cite[Proposition 2.11]{Ric} can carry over verbatim to show that $\Phi$ is in fact a triangulated functor. Also, definition of $\Phi$ implies that it preserves coproducts.
\end{proof}

In the following, we plan to prove that the functor $\Phi$ defined above  is fully faithful.

\begin{lemma}\label{3.5}
There exists an equivalence
$$\psi: \Add \FT \st{\sim}\lrt \Prj \FT$$
of categories, where $\Add \FT$ denotes the full additive subcategory of $\K\K(\Mod\CS)$ formed by all direct summands of direct sums of objects of $\FT$.
\end{lemma}
\begin{proof}
In view of Yoneda lemma, there  is a bijection between $\FT$ and the class of representable functors $\Hom_{\K\K(\Mod \CS)}(-, \BT)$, where $\BT \in \FT$. Each object $\BT$ of $\FT$ satisfies the same property as $({\rm P}1)$. We use this fact to show that this bijection can be extended to the equivalence $\psi: \Add \FT \lrt \Prj \FT$. Indeed, for two families $\{\BT_i\}_{i \in I}$ and $\{\BT_j\}_{j \in J}$ of objects of $\FT$, there exist the following isomorphisms
\\$\begin{array}{llllll}
\Hom_{\K\K(\Mod\FT)}(\oplus_{i \in I}\BT_i, \oplus_{j\in J}\BT_j) & \cong \prod_{i \in I}\Hom_{\K\K(\Mod\CS)}(\BT_i, \oplus_{j\in J}\BT_j)\\
& \cong \prod_{i \in I} \oplus_{j \in J} \Hom_{\K\K(\Mod \CS)}(\BT_i , \BT_j)\\
& \cong \prod_{i \in I} \oplus_{j \in J} \Hom_{\Mod \FT}(\Hom(-, \BT_i) , \Hom(-, \BT_j))\\
& \cong \prod_{i \in I} \Hom_{\Mod \FT}(\Hom(-, \BT_i) , \oplus_{j \in J} \Hom (-, \BT_j))\\
& \cong \Hom_{\Mod \FT}( \oplus_{i \in I} \Hom(-, \BT_i) , \oplus_{j \in J} \Hom (-, \BT_j))\\
& \cong \Hom_{\Mod \FT}(  \Hom(-, \oplus_{i \in I}\BT_i) ,  \Hom(-, \oplus_{j \in J} \BT_j)).
\end{array}$
\\ Now, since every object of $\Add \FT$, resp. $\Prj \FT$, is a direct summand of some objects of the form $\oplus \BT_i$, resp. $\oplus_i \Hom(-, \BT_i)$, we get that $\psi$ is fully faithful. Moreover, it  follows directly from the definition of $\psi$ that $\psi: \Add \FT \lrt \Prj \FT$ is dense.
\end{proof}

Let $\CCD$ be a triangulated category with coproducts. Recall that An object $X$ of $\CCD$ is called compact if for every set $\{ Y_j\}_{j \in J}$ of objects of $\CCD$, every map $X \rt \coprod_{j \in J} Y_j$ factors through a finite coproduct. The full triangulated subcategory of $\CCD$ consisting of all compact objects is denoted by $\CCD^{\rc}$.

Let $\CCL$ be a set of objects of $\CCD$. Then $\CCL$ generates $\CCD$ provided an object $X$ of $\CCD$ is zero if for all $L \in \CCL$, $\CCD(L,X)=0$. A triangulated category $\CCD$ is called compactly generated, if there is a set of compact objects generating  $\CCD$.

\begin{proposition}\label{3.5(2)}
The triangulated category $\K\K_{\FT}(\Mod \CS)$ is compactly generated with $\FT$ as a compact generating set.
\end{proposition}
\begin{proof}
Let $(-, \FT)$ denote the class of all representable functor $\Hom_{\K\K(\Mod \CS)}(-, \BT)$, where $\BT$ runs through objects of $\FT$. Then the equivalence $\psi: \Add \FT \lrt \Prj \FT$ induces an equivalence
$$\bar{\psi}: {\rm Loc} \FT \st{\sim}\lrt {\rm Loc}(-, \FT).$$
It is known that ${\rm Loc}(-, \FT) \simeq \D(\Mod \FT)$ and hence $\K\K_{\FT}(\Mod \CS)={\rm Loc} (-, \FT)$ is compactly generated with $\FT$ as a compact generating set.
\end{proof}

\begin{remark}\label{3.5(3)}
Assume that $\K_{\CT}(\Mod \CS)$ is the smallest full localizing subcategory of $\K(\Mod \CS)$ containing $\CT$. The same argument as above works to show that $\K_{\CT}(\Mod \CS) \simeq \D(\Mod \CT)$ and so $\K_{\CT}(\Mod \CS)$ is compactly generated with $\CT$ as a compact generating set.
\end{remark}
\vspace{0.2cm}

Let $\K_{\CT}^{-, \CT \bb}(\Mod \CS)$ denote the full triangulated subcategory of $\K_{\CT}^-(\Mod \CS)$ consisting of all complexes $X$, in which there is an integer $n=n(X)$ such that $\Hom_{\K_{\CT}(\Mod \CS)}(T , X[i])=0$  for all $i<n$ and  every $T\in \CT$. Also, we denote by $\K\K_{\FT}^-(\Mod \CS)$ the full subcategory of $\K\K_{\FT}(\Mod \CS)$ formed by all bounded above complexes. Moreover, $\K\K_{\FT}^{-, \FT\bb}(\Mod \CS)$ denotes the full triangulated subcategory of $\K\K_{\FT}^-(\Mod \CS)$ formed by all complexes $\X$  in which there is an integer $n=n(\X)$ such that $ \Hom_{\K\K_{\FT}(\Mod \CS)}(\BT, \X[n]) =0$ for all $i<n$ and every $\BT \in \FT$.\\

We have the following characterizations of subcategories of $\K\K_{\FT}(\Mod \CS)$, resp. $\K_{\CT}(\Mod \CS)$.

\begin{lemma}\label{Char1}
\begin{itemize}
\item[$(i)$] Let $\X$ be a complex in $\K\K_{\FT}(\Mod \CS)$. Then $\X$ lies in $\K\K_{\FT}^-(\Mod \CS)$, up to isomorphism, if and only if there exists an integer $N>0$ such that $$\Hom_{\K\K(\Mod \CS)}(\BQ[-n], \X)=0,$$ for all $n\geq N$  and  every complex $\BT$ in $ \FT$.
\item[$(ii)$] Let $\X$ be a complex in $\K_{\CT}(\Mod \CS)$. Then $\X$ lies in $\K^-_{\CT}(\Mod \CS)$, up to isomorphism, if and only if there exists an integer $N>0$ such that $$\Hom_{\K(\Mod \CS)}(\BQ[-n], \X)=0,$$ for all $n \neq N$ and every complex $T$ in $ \CT$.
\end{itemize}
\end{lemma}

\begin{proof}
These statements follow directly from the facts that $\K\K_{\FT}(\Mod \CS)= {\rm Loc}\FT$ and $\K_{\CT}(\Mod \CS)= {\rm Loc}\CT$, respectively.
\end{proof}

\begin{lemma}\label{Ric1}
\begin{itemize}
\item[$(i)$] An object $\X \in \K\K_{\FT}^-(\Mod \CS)$ lies in $\K\K_{\FT}^{-, \FT \bb}(\Mod \CS)$, up to isomorphism, if and only if for each complex $\Y \in \K\K_{\FT}^-(\Mod \CS)$ there exists an integer $N$ such that $\Hom_{\K\K_{\FT}(\Mod \CS)}(\Y , \X[n])=0$, for all $n <N$.
\item[$(ii)$] An object $\X \in \K_{\CT}^-(\Mod \CS)$ lies in $\K_{\CT}^{-, \CT \bb}(\Mod \CS)$, up to isomorphism, if and only if for each complex $\Y \in \K_{\CT}^-(\Mod \CS)$ $~~~~~$  there exists$~~~~$ an integer $~~~~~~~$ $N$ such that $~~~~~~ ~~~~~\Hom_{\K_{\CT}(\Mod \CS)}(\Y , \X[n])=0$, for  all $n <N$.
\end{itemize}
\end{lemma}
\begin{proof}
 We just prove $(i)$. A similar argument works to prove  $(ii)$.
 First note that a modification of the proof of \cite[Lemma 6.1]{Ric} works to prove that, for a small category $\CA$,
 a complex $X$ in $\K^-(\Prj \CA)$ lies in $\K^{-, \bb}(\Prj \CA)$, up to isomorphism, if and only if for every complex $Y \in \K^-(\Prj \CA)$ there is a natural number $N$ such that $\Hom(Y, X[n])=0$, for all $n<N$. Now, the equivalence
 $$\bar{\psi}: {\rm Loc}\FT \lrt {\rm Loc} (-, \FT)$$
 of triangulated categories yields the result.
\end{proof}

\begin{lemma}\label{Ric2}
\begin{itemize}
\item[$(i)$]An object $\X \in \K\K^{-, \FT \bb}_{\FT}(\Mod \CS)$ lies in ${\rm thick} (\oplus \FT )$, up to isomorphism, if and only if for each complex $\Y \in
   \K\K^{-, \FT \bb}_{\FT}(\Mod \CS)$ there is an integer $N$ such that $\Hom_{\K\K_{\FT}(\Mod \CS)}(\X ,\Y[n])=0$, for  all $n >N$.

\item[$(ii)$]An object $\X \in \K^{-, \CT \bb}_{\CT}(\Mod \CS)$ lies in ${\rm thick} ( \oplus \CT ) $, up to isomorphism, if and only if for each complex $~~~~~~~ \Y ~~ \in
    ~~~~ \K^{-, \CT \bb}_{\CT}(\Mod \CS)~~~~~~~~$ there is an integer $N$ $~~~~~~~~$ such that   $\Hom_{\K_{\CT}(\Mod \CS)}(\X ,\Y[n])=0$, for all $n >N$.
\end{itemize}
\end{lemma}

\begin{proof}
$(i)$ Let $\CA$ be a small category. One should apply an argument similar to \cite[Lemma 6.2]{Ric}  verbatim to show that
a complex $X$ in $\K^{-, \bb}(\Prj \CA)$ lies in $\K^{\bb}(\Prj \CA)$, up to isomorphism, if and only if for every complex $Y \in \K^{-, \bb}(\Prj \CA)$, $\Hom(X,Y[n])=0$, for large $n$. Now, the equivalence
$$\bar{\psi}: {\rm Loc}\FT \lrt {\rm Loc} (-, \FT)$$
implies the desired characterization and completes the proof.

The statement $(ii)$ can be obtained from the same argument as above.
\end{proof}

\begin{remark}\label{rem1}
Consider a bijection  $\phi: \CT \lrt \FT$ which takes any complex $T$ in $\CT$ to the stalk complex $\BT$ given by the complex $T$. The map $\phi$ induces an equivalence $\Mod \CT \simeq \Mod \FT$ and hence an equivalence $\D(\Mod \CT) \simeq \D(\Mod \FT)$. Consequently, we have the following equivalence of triangulated categories
$$ \K\K_{\FT}(\Mod \CS) \st{\sim} \rt {\rm Loc}(-, \FT) \st{\sim} \rt \D(\Mod \CS) \st{\sim} \rt \D(\Mod \CT).$$
According to the characterizations that are given in Lemmas \ref{Char1}, \ref{Ric1} and \ref{Ric2}, one can deduce that the above equivalence is restricted to the following equivalences
\begin{equation}\label{diag1}
\xymatrix{\K\K_{\FT}(\Mod \CS) \ar[r]^{\sim} \ar@{<-_)}[d] & \D(\Mod \CT)\ar@{<-_)}[d]
\\ \K\K_{\FT}^-(\Mod \CS) \ar[r]^{\sim}\ar@{<-_)}[d] & \D^-(\Mod \CT)\ar@{<-_)}[d]\\
\K\K_{\FT}^{-, \FT \bb}(\Mod \CS) \ar[r]^{\sim}& \D^{\bb}(\Mod \CT).}
\end{equation}
\end{remark}

\begin{proposition}\label{Phi-equi}
The triangulated functor $\Phi: \K \K_{\FT}(\Mod \CS) \lrt \K(\Mod \CS)$ is full and faithful. Moreover, it induces the equivalence  $$\Phi: \K\K_{\FT}(\Mod \CS) \lrt \K_{\CT}(\Mod \CS),$$
of triangulated categories.
\end{proposition}

\begin{proof}
In view of Proposition \ref{3.5(2)} and Remark \ref{3.5(3)}, $\K\K_{\FT}(\Mod \CS)$ and $\K_{\CT}(\Mod \CS)$ are compactly generated with compact generating sets $\FT$ and $\CT$, respectively and by Lemma 2.2 of \cite{N92}, $\K\K_{\FT}(\Mod \CS)^{\rm c}={\rm thick}(\FT)$ and $\K_{\CT}(\Mod \CS)^{\rm c}={\rm thick}(\CT ) $. So,
by \cite[Proposition 6]{Mi2}, it suffices to show that the restriction of $\Phi$ to compact objects is an equivalence.   According to Construction \ref{Cons}, ${\rm thick}(\CT) = \bigcup_{n \in \N} \lan \CT \ran_n$. Now, we show that $\Phi|: {\rm thick}(\FT) \lrt {\rm thick}(\CT)$ is full and faithful, using induction on $n$

First assume that $n=1$. Properties $({\rm P}1)$ and $({\rm P}2)$ imply that if $\X$ and $\Y$ belong to $\lan \FT  \ran_1$, then
$$ \Hom_{\K\K_{\FT}(\Mod \CS)}(\X, \Y) \cong \Hom_{\K_{\CT}(\Mod \CS)}(\Phi(\X), \Phi(\Y)).$$
Now let $n>1$ and $\X$ and $\Y$ belong to $\lan \FT \ran_n$.  Hence there exist triangles
$$\X' \rt \X \rt \X'' \rightsquigarrow,$$
$$\Y' \rt \Y \rt \Y '' \rightsquigarrow,$$
where $\X' \in \lan \FT \ran_m, \Y' \in \lan \FT\ran_m', \X'' \in \lan \FT \ran_{l'}$ and $ \Y'' \in \lan \FT  \ran_l$ such that $m, m', l, l'<n$. So, by induction  hypothesis, we have the following isomorphism
$$ \Hom_{\K\K_{\FT}(\Mod \CS)}(\X, \Y) \cong \Hom_{\K_{\CT}(\Mod \CS)}(\Phi(\X), \Phi(\Y)).$$
Consequently, $\Phi| : \K\K_{\FT}(\Mod \CS)^{\rm c} \lrt \K_{\CT}(\Mod \CS)^{\rm c}$ is full and faithful. Clearly, $\Phi|$ is also dense and so is an equivalence.
\end{proof}

Let $\CCD$ be a triangulated category with direct sums. Let $\CCL$ be a set of objects of $\CCD$. Then $\oplus \CCL$ denotes the set of all direct sums of objects of $\CCL$.

\begin{proposition}\label{restrict}
There is a commutative diagram
\[\xymatrix{ \K\K_{\FT}(\Mod \CS) \ar[r]^{\st{\Phi} \sim} \ar@{<-_{)}}[d] & \K_{\CT}(\Mod \CS) \ar@{<-_{)}}[d] \\
\K\K^{-}_{\FT}(\Mod \CS) \ar[r]^{\sim} \ar@{<-_{)}}[d] &  \K^{-}_{\CT}(\Mod \CS) \ar@{<-_{)}}[d] \\
\K\K_{\FT}^{- ,\FT\bb}(\Mod \CS) \ar[r]^{\sim} \ar@{<-_{)}}[d] & \K^{-, \CT \bb}_{\CT} (\Mod \CS) \ar@{<-_{)}}[d] \\
 {\rm thick}(\oplus \FT) \ar@{<-_{)}}[d] \ar[r]^{\sim}  & {\rm thick}(\oplus \CT ) \ar@{<-_{)}}[d] \\
{\rm thick}(  \FT ) \ar[r]^{\sim} & {\rm thick}( \CT )
}\]
of triangulated categories  whose rows are triangle equivalences.
\end{proposition}
\begin{proof}
First observe that every triangle equivalence induces an equivalence on the compact objects. Hence, in view of Proposition \ref{3.5(2)} and Remark \ref{3.5(3)}, we have the last induced equivalence.

The other induced equivalences follow directly from Lemmas \ref{Char1}, \ref{Ric1} and \ref{Ric2}.
\end{proof}

\begin{remark}
Observe that the condition $({\rm P}3)$ of \ref{Properties} has been used to define the functor $\Phi$. Indeed, let $\X$ be a complex in $\K\K_{\FT}(\Mod \CS)$. Then we can associate an object $\{\X^{**}, d_i\}$ of $G(\CS)$ to $\X$. Now,  condition $({\rm P}3)$ guarantees that the sum $\Sigma d_i$ is well-defined. So, we can form the total complex of $\X^{**}$, with $\Sigma d_i$ as the differentials. Furthermore, to construct the functor $\Phi: \K\K_{\FT}(\Mod \CS) \lrt \K(\Mod \CS)$ it is not necessary to assume that $\CT \subseteq \K(\Prj \CS)$. We have just needed  properties $({\rm P}1)$, $({\rm P}2)$ and $({\rm P}3)$ to define $\Phi$ and to prove that it is full and faithful.
\end{remark}

\s {\sc Relative derived category.}\label{reldercat}
Let $\CS$ be a small category and $\CX$ be a full subcategory of $\K(\Mod \CS)$. A complex $\Y$ in $\K(\Mod \CS)$ is called $\CX$-acyclic, if for every object $\X$ of $\CX$ and all $i \in \Z$, $\Hom_{\K(\Mod \CS)}(\X, \Y[i])=0$. Let $\K_{\CX \mbox{-} {\rm ac}}(\Mod \CS)$ denote the full triangulated subcategory of $\K(\Mod \CS)$ consisting of all $\CX$-acyclic complexes. Clearly, $\K_{\CX \mbox{-} {\rm ac}}(\Mod \CS)$ is a thick subcategory of $\K(\Mod \CS)$. So we get the triangulated category
$$\D_{\CX}(\Mod \CS):= \K(\Mod \CS)/\K_{\CX\mbox{-} {\rm ac}}(\Mod \CS)$$
which is called the relative derived  category of $\Mod \CS$ with respect to $\CX$.\\

Recall that if $\CCD'$ is a triangulated subcategory of $\CCD$, then the left and right orthogonal of $\CCD'$ in $\CCD$ are given  by
\[ ^\perp\CCD'=\{ X \in \CCD \ | \ \Hom_{\CCD}(X,Y)=0, \ {\rm{for \ all}} \ Y \in \CCD' \},\]
\[ \CCD'^\perp=\{ X \in \CCD \ | \ \Hom_{\CCD}(Y,X)=0, \ {\rm{for \ all}} \ Y \in \CCD' \}.\]

\begin{lemma}\label{lem6}
There is an equivalence
$$\D_{\CT}(\Mod \CS) \simeq \K_{\CT}(\Mod \CS)$$
of triangulated categories.
\end{lemma}
\begin{proof}
Consider the following sequence of functors
$$ \K_{\CT}(\Mod \CS) \st{\iota} \hookrightarrow \K(\Mod \CS) \st{Q} \rt \D_{\CT}(\Mod \CS),$$
where $\iota$ is the inclusion and $Q$ is the canonical functor. Since $\K_{\CT}(\Mod \CS) = {\rm Loc} \CT  $, $\K_{\CT}(\Mod \CS) \subseteq {}^\perp \K_{\CT\mbox{-}{\rm ac}}(\Mod \CS)$. Hence, by  Lemma 9.1.5 of \cite{Ne01}, the composition functor
$Q \circ \iota: \K_{\CT}(\Mod \CS) \lrt \D_{\CT}(\Mod \CS)$ is full and faithful. So to complete the proof, it remains to show that $Q\circ \iota$ is dense.

Since, by Remark \ref{3.5(3)}, $\K_{\CT}(\Mod \CS)={\rm Loc}\CT$ is compactly generated, Theorem 4.1 of \cite{Ne96} implies that the inclusion
$$\K_{\CT}(\Mod \CS) \st{\iota}\hookrightarrow \K(\Mod \CS)$$
has a right adjoint. So, for every complex $\X$ in $\K(\Mod \CS)$
there is a triangle
$$\X' \rt \X \rt \X'' \rightsquigarrow,$$
in $\K(\Mod \CS)$, where $\X' \in \K_{\CT}(\Mod \CS)$ and $\X'' \in (\K_{\CT}(\Mod \CS))^\perp$. Since $\CT$ is contained in $\K_{\CT}(\Mod \CS)$, $\X''$ is a $\CT$-acyclic complex. Therefore, if we consider the above triangle in $\D_{\CT}(\Mod \CS)$ under the canonical functor $Q$, we get  that $\X$ is isomorphic to $\X'$. This means that $Q \circ \iota$ is dense.
\end{proof}

\begin{theorem}\label{relativeq}
As above, let $\CT$ be a set of objects in $\K(\Mod \CS)$ satisfying conditions $({\rm P}1)$, $({\rm P}2)$ and $({\rm P}3)$ of \ref{Properties}. Then there is the following commutative diagram of triangulated categories with equivalence rows
\[\xymatrix{ \D(\Mod \CT) \ar[r]^{ \sim} \ar@{<-_{)}}[d] & \D_{\CT}(\Mod \CS) \ar@{<-_{)}}[d] \\
\D^{-}(\Mod \CT) \ar[r]^{\sim} \ar@{<-_{)}}[d] &  \D^{-}_{\CT}(\Mod \CS) \ar@{<-_{)}}[d] \\
\D^{\bb}(\Mod \CT) \ar[r]^{\sim} \ar@{<-_{)}}[d] & \D^{\bb}_{\CT} (\Mod \CS) \ar@{<-_{)}}[d] \\
\K^{\bb}(\Prj \CT) \ar@{<-_{)}}[d] \ar[r]^{\sim}  & {\rm thick}( \oplus \CT ) \ar@{<-_{)}}[d] \\
\K^{\bb}(\prj \CT) \ar[r]^{\sim} & {\rm thick}( \CT )
}\]
\end{theorem}

\begin{proof}
By Remark \ref{rem1} there is an equivalence $ \D(\Mod \CT) \simeq \K\K_{\FT}(\Mod \CS)$. Moreover, Proposition \ref{Phi-equi} implies that $\K\K_{\FT}(\Mod \CS) \simeq \K_{\CT}(\Mod \CS)$. Hence, there exists the first equivalence thanks to Lemma \ref{lem6}.
For the second equivalence, consider
the canonical functor $\K_{\CT}^-(\Mod \CS) \st{Q|} \lrt \D_{\CT}^-(\Mod \CS)$. As above, \cite[Lemma 9.1.5]{Ne01} implies that $Q|$ is full and faithful. Hence, we just need to prove that $Q|$ is dense as well.

Let $\X$ be a complex in $\K^-(\Mod \CS)$. As in the proof of Lemma \ref{lem6}, there exists a triangle
$$ \X' \rt \X \rt \X'' \rightsquigarrow,$$
where $\X' \in \K_{\CT}(\Mod \CS)$ and $\X''$ is a $\CT$-acyclic complex.

Since $\X \in \K^-(\Mod \CS)$, for each complex $T$ in ${\rm thick}(\CT)$, there is an integer $N=N(T)$ such that $\Hom_{\K(\Mod \CS)}(T[-n], \X)=0$, for all $n \geq N$. Also, $\Hom_{\K(\Mod \CS)}(T[n], \X)=0$, for all $n \in \Z$ and all $T \in {\rm thick}(\CT)$.  Thus, for every complex $T \in {\rm thick}(\CT)$ there is an integer $N=N(T)$ such that $\Hom_{\K(\Mod \CS)}(T[-n], \X')=0$,  for all $n \geq N$. Hence, by Lemma \ref{Char1}, $\X'$ belongs to $\K^-_{\CT}(\Mod \CS)$. The image of the  above triangle in $\D_{\CT}^-(\Mod \CS)$  implies that $\X \cong \X'$ in $\D_{\CT}^-(\Mod \CS)$ and so $Q|: \K^-_{\CT}(\Mod \CS) \lrt \D^-_{\CT}(\Mod \CS)$ is an equivalence. Now, Diagram \ref{diag1} in Remark \ref{rem1} and Proposition \ref{restrict} yield the desired equivalence $\D^-(\Mod \CT) \lrt \D^-_{\CT}(\Mod \CS)$ that commutes the first square.

The same argument as above together with Lemma \ref{Ric1} can be applied to show that the canonical functor $Q|: \K_{\FT}^{-, \FT \bb}(\Mod \CS) \st{\sim} \lrt \D_{\FT}^{\bb}(\Mod \CS)$ is an equivalence. Thus the second square follows from Diagram \ref{diag1} and Proposition \ref{restrict}.

Similarly, the characterization of ${\rm thick}(\oplus \CT)$, Lemma \ref{Ric2}, in conjunction with Proposition \ref{restrict} imply the equivalence $\K^{\bb}(\Prj \CT) \lrt {\rm thick}(\oplus \CT)$ making the diagram commutative.

The last equivalence follows from the fact that every equivalence between  triangulated categories can be restricted to the equivalence between their compact objects.
\end{proof}

Let $R$ be a commutative ring. A small category $\CS$ is called $R$-flat, if $\CS(x,y)$ is flat $R$-module, for every $x, y \in \CS$.
Keller \cite[9.2, Corollary]{Kel} proved that two $R$-flat categories $\CS$ and $\CS'$ are derived equivalent, i.e. $\D(\Mod \CS) \simeq \D(\Mod \CS')$, if and only if there exists a special subcategory $\CT$ of $\K^{\bb}(\prj \CS)$, called tilting subcategory for $\CS$, such that $\CT$ is equivalent to $\CS'$.

In the following theorem we provide a sufficient condition for derived equivalences of functor categories without flatness assumption on the categories involved in the derived equivalence.

\begin{theorem}\label{KelThm}
As in setup \ref{Properties}, let $\CT$ be a set of objects  of $\K(\Mod \CS)$  satisfying properties  $({\rm P}1)$, $({\rm P}2)$ and $({\rm P}3)$. Assume that a complex $\X$ in $\K(\Mod \CS)$ is acyclic if and only if it is $\CT$-acyclic. Then there exists the following commutative diagram
\[\xymatrix{ \D(\Mod \CT) \ar[r]^{ \sim} \ar@{<-_{)}}[d] & \D(\Mod \CS) \ar@{<-_{)}}[d] \\
\D^{-}(\Mod \CT) \ar[r]^{\sim} \ar@{<-_{)}}[d] &  \D^{-}(\Mod \CS) \ar@{<-_{)}}[d] \\
\D^{\bb}(\Mod \CT) \ar[r]^{\sim} \ar@{<-_{)}}[d] & \D^{\bb} (\Mod \CS) \ar@{<-_{)}}[d] \\
\K^{\bb}(\prj \CT) \ar[r]^{\sim} & \K^{\bb}(\prj \CS)
}\]
in which rows are equivalences of triangulated categories.
\end{theorem}
\begin{proof}
Definition of the relative derived category  in conjunction with our assumption on acyclic complexes imply  that $\D_{\CT}(\Mod \CS)$ coincides with $\D(\Mod \CS)$. Hence, in view of Theorem \ref{relativeq}, we have the desired diagram.
\end{proof}

\begin{remark}\label{rem}
Our proofs show that the above results in this section can be extended to a skeletally small category $\CS$.
\end{remark}

\s Let $A$ be an artin $k$-algebra, where $k$ is a  commutative artinian ring. Then $A$ is called Gorenstein if ${\rm id} \ {}_{A} A< \infty$ and ${\rm id} \ A_{A} < \infty $. For a class $\CX$ of $A$-modules, the left and right orthogonals of $\CX$ are defined as follows
\[{}^\perp\CY:=\{ M \in \Mod A \ | \ \Ext^1_{A}(M,Y)=0, \ {\rm{for \ all}} \ Y \in \CY \}\]
and
\[\CX^\perp:=\{ M \in \Mod A \ | \ \Ext^1_{A}(X,M)=0, \ {\rm{for \ all}} \ X \in \CX \}.\]
The artin algebra $A$ is called virtually Gorenstein if $(\GPrj A)^\perp = {}^\perp (\GInj A)$.\\

Let $A$ and $B$ be two rings. Then $A$ and $B$ are called derived equivalent if there exists a triangle equivalence $\D^{\bb}(\Mod A) \simeq \D^{\bb}(\Mod B)$. If $A$ and $B$ are right coherent, then $A$ and $B$ are derived equivalent provided there exists a triangle equivalence $\D^{\bb}(\mmod A) \simeq \D^{\bb}(\mmod B)$.

Let $A$  be a left coherent ring. It is known, by \cite[Theorem 1.1]{N08}, that $\K(\Prj A)$ is compactly generated and $\K^{\rm c}(\Prj A) \simeq \D^{\bb}(\mmod A^{\op})^{\op}$. Moreover, every equivalence between triangulated categories can be restricted to an equivalence between their compact objects. Hence, if $A$ and $B$ are two right and left coherent rings such that $\K(\Prj A) \simeq \K(\Prj B)$, then $A$ and $B$ are derived equivalent. In the following theorem we plan to prove the converse for general rings.

\begin{theorem}\label{Thmder}
Let $A$ and $B$ be two rings that are derived equivalent. Then there is a commutative diagram
\[\xymatrix@C-1.pc@R-1pc{ & \K(\Prj B) \ar[rr]^{\al_1} \ar@{<-_{)}}[dl] \ar@{<-_{)}}[dd] && \K(\Prj A) \ar@{<-_{)}}[dl] \ar@{<-_{)}}[dd]\\
\K(\prj B)\ar@{<-_{)}}[dd] \ar@{^{(}->}[rr]^{\ \ \ \ \ \ \al_2} && \K(\prj A)\ar@{<-_{)}}[dd]  &\\
& \K_{\rm tac}(\Prj B) \ar[rr]^{\beta_1 \ \ \ \ \ } \ar@{<-_{)}}[dl]  && \K_{\rm tac}(\Prj A) \ar@{<-_{)}}[dl] \\
\K_{\rm tac}(\prj B) \ar@{^{(}->}[rr]^{\beta_2} && \K_{\rm tac}(\prj A)&
}\]
in which $\al_1$ and $\beta_1$ are triangulated equivalences. Moreover, if $A$ and $B$ are virtually Gorenstein algebras, then $\beta_2$ is also an equivalence of triangulated categories.
\end{theorem}
\begin{proof}
By Theorem 4.6 of \cite{Ric} \ there is a tilting complex \ $T$ in  \ $\K^{\bb}(\prj A)$ \ \ such that $\ \ \ \ \End_{\K(\prj A)}(T) \cong B$. Consider the full subcategory $\K(\Add T)$ of $\K\K(\Mod A)$. Observe that an object $(\X^{**}, d, \delta)$ of $\K(\Add T)$ is a complex in $\K\K(\Mod A)$ such that for each $j \in \Z$, $(X^{*j}, d)$ is a direct summand of $\oplus_{‎‏i \in I} T_i$, where $T_i \cong T$.
In a similar way as in Proposition \ref{functor}, we can construct a triangle  functor $\Phi: \K(\Add T) \lrt \K(\Prj A)$ that preserves coproducts.

On the other hand, the functor $\Hom_{\K(\prj A)}(T,-)$ provides an equivalence $\K(\Add T) \simeq \K(\Prj B)$. Hence $\K(\Add T)$ is a compactly generated triangulated category. Moreover, by \cite[Proposition 7.12]{N08}, a complex  in $\K(\Add T)$ is compact if and only if it is isomorphic to a complex $\X$ satisfying
\begin{itemize}
\item[$(i)$] $\X$ is a complex with terms in $\add T$,
\item[$(ii)$] $X^n=0$, for $n \ll 0$,
\item[$(iii)$] $\Hom_{\K(\Add T)}(\X, T[n])=0$, for $n \ll 0$.
\end{itemize}
It is known that if $T$ is a tilting complex over a ring $A$, then $T^*:= \Hom_A(T,A)$ is a tilting complex over $A^{\op}$. Similarly, we have a triangulated functor $\Phi^*: \K(\Add T^*) \lrt \K(\Prj A^{\op})$ that preserves coproducts.

Now, one just should use definitions of the functors $\Phi$ and $\Phi^*$ to obtain the following commutative diagram of triangulated categories
\[\xymatrix{\K(\Add T)^{\rm c} \ar[r]^{\Phi|} \ar[d]_{\widehat{\Hom}(-,A)} & \K(\Prj A)^{\rm c}  \\
 \K^{-, T^* \bb}(\Add T^*)  \ar[r]^{\Phi^*|} & \K^{-,\bb}(\Prj A^{\op})\ar[u]_{\Hom(-,A)}. }\]
 Note that the functor $\widehat{\Hom}(-,A): \K(\Add T)^{\rm c} \lrt \K^{-, T^* \bb}(\Add T^*)$ maps any complex $(\X^{**},d , \delta)$ of $\K(\Add T)^{\rm c}$ to the complex $((\X^{**}, A), d^*, \delta^*)$ of $\K^{-, T^*\bb}(\Add T^*)$.

 Since $\Phi^*$, $\widehat{\Hom}(-,A)$ and $\Hom(-,A)$ are equivalences, $\Phi|$ must be an equivalence. Moreover, Proposition 6 of \cite{Mi2} implies that $\Phi: \K(\Add T) \lrt \K(\Prj A)$ is an equivalence. Consequently, there exists an equivalence
$$\al_1: \K(\Prj B) \st{\sim} \lrt \K(\Prj A).$$

To prove that $\beta_1=\al_1|$ is an equivalence, it  is enough to show  that $\al_1$ maps  any complex in $\K_{\tac}(\Prj B)$ to a complex of $\K_{\tac}(\Prj A)$ and moreover, show that $\beta_1$ is dense.
First observe that by Theorem \ref{relativeq}, the equivalence $\al_1$ can be restricted to the triangle equivalence
$$\al_1|: \K^{\bb}(\Prj B) \lrt \K^{\bb}(\Prj A).$$
So, if $P$ is a projective $A$-module, then there is a bounded complex $\BQ$ of projective $B$-modules such that $\al_1(\BQ)\cong P$.

Let $\X$ be a complex in   $\K_{\tac}(\Prj B)$. Then for each $P \in \Prj A$ and each $i \in \Z$, there is the following isomorphisms
\[ \begin{array}{ll}
\Hom_{\K(\Prj A)}(P, \al_1(\X)[i]) & \cong \Hom_{\K(\Prj A)}(\al_1(\BQ), \al_1(\X)[i])\\
&\cong \Hom_{\K(\Prj B)}(\BQ, \X[i])
\end{array}\]
and
\[ \begin{array}{ll}
\Hom_{\K(\Prj A)}( \al_1(\X)[i], P) & \cong \Hom_{\K(\Prj A)}( \al_1(\X)[i], \al_1(\BQ))\\
&\cong \Hom_{\K(\Prj B)}( \X[i], \BQ),
\end{array}\]
where $\BQ \in \K^{\bb}(\Prj B)$.
Since $\X \in \K_{\tac}(\Prj B)$, for every $Q \in \Prj B$ and all $i \in \Z$
$$\Hom_{\K(\Prj B)}(Q, \X[i])=0=\Hom_{\K(\Prj B)}(\X[i],Q).$$
So,  using an induction  argument  on the length of  $\BQ$  one can deduce that
$$\Hom_{\K(\Prj A)} \ (P \ , \ \al_1(\X)[i])=0=\Hom_{\K(\Prj A)}(\al_1(\X)[i] , P),$$
 where $P \in \Prj A$ and $i \in \Z$.
A similar  argument as above implies that $\beta_2$ is dense.

Furthermore, it is clear that there exist the induced functors $\al_2$ and $\beta_2$ that are full and faithful.
In case $A$ and $B$ are virtually Gorenstein algebras, then, by \cite[Theorem 8.2]{Be},  $\K_{\tac}^{\rm c}(\Prj A) \simeq \K_{\tac}(\prj A)$ and $\K_{\tac}^{\rm c}(\Prj B)\simeq \K_{\tac}(\prj B)$. Therefore, $\beta_2$ is an equivalence.
\end{proof}

\begin{remark}
Let $A$ and $B$ be two right and left coherent rings that are derived equivalent. By \cite[Proposition 9.1]{Ric}, there is an equivalence $\Psi: \D^{\bb}(\mmod A^{\op}) \st{\sim}\lrt \D^{\bb}(\mmod B^{\op})$, and so $\Psi^{\op}: \D^{\bb}(\mmod A^{\op})^{\op} \st{\sim}\lrt \D^{\bb}(\mmod B^{\op})^{\op}$ is an equivalence.
The equivalence $\Phi: \K(\Prj A) \lrt \K(\Prj B)$, that is proved in Theorem \ref{Thmder}, can be viewed as an extension of the equivalence $\Psi^{\op}$. In fact, by a result of Neeman \cite[Theorem 1.1]{N08}, $\K(\Prj A)$, resp. $\K(\Prj B)$, is compactly generated and $\K^{\rm c}(\Prj A) \simeq \D^{\bb}(\mmod A^{\op})^{\op}$, resp. $\K^{\rm c}(\Prj B) \simeq \D^{\bb}(\mmod B^{\op})^{\op}$. Moreover, these maps fit into the following commutative diagram
\[ \xymatrix{ \K^{\rm c}(\Prj A) \ar[r]^{\Phi|} \ar[d]^{\wr} & \K^{\rm c}(\Prj B) \ar[d]^{\wr}\\
\D^{\bb}(\mmod A^{\op})^{\op} \ar[r]^{\Psi^{\op}} & \D^{\bb}(\mmod B^{\op})^{\op}.}\]
\end{remark}
\vspace{0.2cm}

Let $\underline{\Gprj}A$ denote the stable category of $\Gprj A$ modulo the full  subcategory $\prj A$.
Beligiannis \cite[Theorem 8.11]{Be} proved that if $A$ and $B$ are derived equivalent finite dimensional algebras, then there is a triangle equivalence $\underline{\Gprj}A \simeq \underline{\Gprj}B$. Also, Kato \cite[Theorem 3.8]{Ka} showed that if $A$ and $B$ are derived equivalent right and left coherent rings, and if ${\rm inj} \dim {}_{A}A<\infty$ or ${\rm inj} \dim A_A<\infty$, then there is an equivalence $\underline{\Gprj}A \simeq \underline{\Gprj}B$ of triangulated categories.\\

By using Theorem \ref{Thmder}, we can prove this result for virtually Gorenstein algebras that are derived equivalent.

\begin{corollary}
Let $A$ and $B$ be two rings that are derived equivalent. Then there is a triangle  equivalence
$\underline{\GPrj}A \simeq \underline{\GPrj}B$. If $A$ and $B$ are virtually Gorenstein algebras, then  $\underline{\Gprj}A \simeq \underline{\Gprj}B$ as triangulated categories.
\end{corollary}

\begin{proof}
The result follows from the known equivalences $\K_{\tac}(\Prj A) \simeq \underline{\GPrj}A$ and $\K_{\tac}(\prj A) \simeq \underline{\Gprj}A$.
\end{proof}

An artin algebra $\La$ is said to be of finite Cohen-Macaulay
type (finite CM-type, for short), if there are only finitely many indecomposable
finitely generated Gorenstein projective $\La$-modules, up to isomorphism.
Let $A$ and $B$ be Gorenstein artin algebras that are derived equivalent. Then $A$ is of finite CM-type if and only if $B$ is so; see \cite[Theorem 4.6]{Hap2} and \cite[Propostion 3.10]{P}. The following corollary extends this result to virtually Gorenstein algebras.

\begin{corollary}
Let $A$ and $B$ be two virtually Gorenstein algebras that are derived equivalent. Then $A$ is of finite CM-type if and only if $B$ is so.
\end{corollary}

\begin{remark}
 In \cite{J2} J{\o}rgensen proves that if  $A$ is  a commutative  noetherian ring with a dualizing complex, then  the inclusion functor $\K_{\tac}(\Prj A) \hookrightarrow \K(\Prj A)$ admits a right adjoint. This result was generalized in \cite{MS} to commutative noetherian rings of finite Krull dimension. Moreover, it is easy to check that, for an arbitrary ring $A$, the existence of a right adjoint for the inclusion $\K_{\tac}(\Prj A)\hookrightarrow \K(\Prj A)$ implies the existence of Gorenstein projective precovers over $A$. For instance, it is proved for commutative  noetherian rings with a dualizing complex by J{\o}rgensen \cite{J2}.

 On the other hand, Theorem \ref{Thmder} implies that
 if $A$ and $B$ are two rings that are derived equivalent, then the inclusion functor $\K_{\tac}(\Prj A) \hookrightarrow \K(\Prj A)$ has a right adjoint if and only if the inclusion functor $\K_{\tac}(\Prj B) \hookrightarrow \K(\Prj B)$ has a right adjoint. Hence, if $A$ is a commutative noetherian ring of finite Krull dimension and $B$ is a ring such that $\D^{\bb}(\Mod A) \simeq \D^{\bb}(\Mod B)$, then Gorenstein projective $B$-modules is a precovering class of $\Mod B$.
\end{remark}

\subsection{Applications}
As an application of our results, in this subsection, we show that over noetherian rings, certain derived equivalences imply Gorenstein derived equivalences.
Throughout this subsection, all rings are noetherian. To present our results we need to fix some notations.

Let us first recall briefly the definition of Gorenstein derived categories. Let $A$ be an artin algebra. A complex $\X $ of finitely generated $A$-modules is called $\Gp$-acyclic if for every $G \in \Gprj A$, the induced complex $\Hom_{A}(G, \X)$ is acyclic. We denote by $\K^{\bb}_{\Gp\mbox{-}\ac}(\mmod A)$ the class of all $\Gp$-acyclic complexes in $\K^{\bb}(\mmod A)$.
The bounded Gorenstein derived category of $\mmod A$, denoted by $\D^{\bb}_{\Gp}(\mmod A)$, is the quotient category $\K^{\bb}(\mmod A)/ \K^{\bb}_{\Gp\mbox{-}\ac}(\mmod A)$. The Gorenstein derived category was studied by Gao and Zhang \cite{GZ}. Recently this category has been studied more in \cite{ABHV, AHV}.

\sss A complex $\X$ in $\D^-(\mmod A)$ is called of finite projective dimension if the functor  $\Hom_{\D(\mmod A)}(\X[i], -)$ vanishes on $\mmod A$, for $i \ll 0$.
The full subcategory of $\D^-(\mmod A)$ consisting of all complexes of finite projective dimension is denoted by $\D^{\bb}(\mmod A)_{\rm fpd}$. It is known that there exists an equivalence
$$Q: \K^-(\prj A) \st{\sim}\lrt \D^-(\mmod A),$$
where $Q$ is the canonical functor. Moreover, this equivalence induces the equivalence
$$\K^{\bb}(\prj A) \st{\sim}\lrt \D^{\bb}(\mmod A)_{\rm fpd}$$
of triangulated categories.

The singularity category of $A$, denoted by $\D^{\bb}_{\sg}(A)$, is then the Verdier quotient $$\D^{\bb}(\mmod A)/\D^{\bb}(\mmod A)_{\rm fpd}.$$

A complex $\X$ in $\D^{\bb}(\mmod A)$ is called of finite Gorenstein projective dimension if the functor $\Hom_{\D(\mmod A)}(\X[i], -)$ vanishes on $\prj A$ for $i \ll0$. The full triangulated subcategory of $\D^{\bb}(\mmod A)$ consisting of all complexes of finite Gorenstein projective dimension is denoted by $\D^{\bb}(\mmod A)_{\rm fGd}$.

Let $\K^{-, \bb}_{\rm t}(\prj A)$ denote the full triangulated subcategory of $\K^{-, \bb}(\prj A)$ formed by all complexes $\X$ such that there is an integer $m=m(\X)$ in which $\Ker \partial^i_{\X} \in \Gprj A$, for all $i \leq m$.

\begin{slemma}\label{Kato}
Let $A$ be a noetherian ring. Then for a complex $\X \in \D^{\bb}(\mmod A)$ the following statements are equivalent
\begin{itemize}
\item[$(i)$] $\X \in \D^{\bb}(\mmod A)_{\rm fGd}$.
\item [$(ii)$] There exists a complex $\Y \in \K^{\bb}(\Gprj A)$ such that $\X\cong \Y$  in $\D(\mmod A)$.
\item[$(iii)$] $\X$ lies in $\K^{-, \bb}_t(\prj A)$, up to isomorphism.
\end{itemize}
\end{slemma}

\begin{proof}
$(i) \Leftrightarrow (ii)$. It follows from \cite[Proposition 2.10]{Ka}.

$(ii) \Rightarrow (iii)$. Note that for every Gorenstein projective $A$-module, there exists a projective resolution that belongs to $\K^{-, \bb}_t(\prj A)$. Now, the assertion  follows by an induction on the length of $\Y$.

$(iii) \Rightarrow (ii)$. First we assume that $X^i=0$ for all $i>0$. Let $m=m(\X)$ be an integer such that $\Ker \partial_{\X}^i \in \Gprj A$ for all $i \leq m$. Then $\X$ is quasi-isomorphic to the complex
$$ 0 \lrt \Ker \partial_{\X}^m \hookrightarrow X^m \st{\partial^m}\lrt X^{m+1} \lrt \cdots \lrt X^{-1} \st{\partial^{-1}} \lrt X^0 \lrt 0$$
that  belongs to $\K^{\bb}(\Gprj A)$
\end{proof}

Let $A$ be a noetherian ring. By \cite{Hap2}, there is  a full and faithful functor $H: \underline{\Gprj}A \lrt \D^b_{\sg}(A)$ such that the diagram
\[\xymatrix{\Gprj A \ar[d]^{\inc} \ar[r]^{\can} & \underline{\Gprj}A \ar[d]^{H} \\ \D^b(\mmod A)  \ar[r]^{\can} & \D^b_{\sg}( A)} \]
is commutative.

It is proved in \cite{Av} that the functor $H: \underline{\Gprj} A \lrt \D^{\bb}_{\sg}(A)$ induces an equivalence $$\underline{\Gprj}A \simeq \D^{\bb}(\mmod A)_{\rm fGd}/ \D^{\bb}(\mmod A)_{\rm fpd}$$ of triangulated categories.

\begin{sdefinition}\label{fGd}
Let $A$ and $B$ be noetherian rings. We say that $A$ and $B$ are ${\rm fGd}$-derived equivalent, if there exists an equivalence $F: \D^b(\mmod A) \lrt \D^b(\mmod B)$ inducing an equivalence $F|: \D^{\bb}(\mmod A)_{\rm fGd} \lrt \D^{\bb}(\mmod B)_{\rm fGd}$ making the following diagram commutative
\[\xymatrix{\D^{\bb}(\mmod A)_{\rm fGd} \ar@{^(->}[r] \ar[d]^{F|} & \D^{\bb}(\mmod A) \ar[d]^{F} \\ \D^{\bb}(\mmod B)_{\rm fGd} \ar@{^(->}[r]  & \D^{\bb}(\mmod B). } \]
\end{sdefinition}

Kato \cite{Ka} constructed examples of such derived equivalences. Let $A$ and $B$ be left and right coherent rings that are derived equivalent. He proved that if  either ${\rm inj} \dim {}_A A < \infty$ or ${\rm inj} \dim  A_A < \infty$, then $\D^{\bb}(\mmod A)_{\rm fGd}$ is equivalent to $\D^{\bb}(\mmod B)_{\rm fGd}$ as triangulated categories.
Also, if $A$ and $B$ are derived equivalent finite dimensional $k$-algebras over field k, then $\D^{\bb}(\mmod A)_{\rm fGd}\simeq\D^{\bb}(\mmod B)_{\rm fGd}$ as well. Using Theorem \ref{Thmder}, we have the following result that gives another example of such equivalences.

\begin{sproposition}\label{Stder}
Let $A$ and $B$ be virtually Gorenstein algebras that are derived equivalent,  via say $\Phi$. Then $\Phi$ is a ${\rm fGd}$-derived equivalence.
\end{sproposition}

\begin{proof}
By Lemma \ref{Kato}, a complex $\X$ belongs to $\D^{\bb}(\mmod A)_{\rm fGd}$ if and only if $\X$ belongs to $\K^{-, \bb}_{\rm t}(\prj A)$, up to isomorphism. So it is enough to show that there exists the following commutative diagram
\[\xymatrix{\K^{-, \bb}_{\rm t}(\prj A) \ar@{^(->}[r] \ar[d]^{\Phi|} & \D^{\bb}(\mmod A) \ar[d]^{\Phi} \\ \K^{-, \bb}_{\rm t}(\prj B) \ar@{^(->}[r]  & \D^{\bb}(\mmod B), } \]
such that vertical maps are equivalences.
Let $\X$ be a complex in $\K^{-, \bb}_{\rm t}(\prj A)$. Then there exists an integer $m$ such that ${\rm H}^i(\X)=0$ and $\Ker \partial_{\X}^i\in \Gprj A$ for all $i \leq m$. We may assume that $m<0$.

Consider  the triangle
\[ (*) \ \ \ {}_{\sqsubset_m} \X \rt \X \rt \X {}_{{}_{m-1} \sqsupset }\rightsquigarrow\]
in $\K(\prj A)$. Since $\Ker \partial_{\X}^m \in \Gprj A$, there is a complete projective resolution
\[ \xymatrix{\BQ: & \cdots \ar[r] & Q^{-2} \ar[r] & Q^{-1} \ar[rr]\ar[dr]  && Q^0 \ar[r] & Q^1 \ar[r] & \cdots \\
&&&& \Ker \partial_{\X}^m \ar[ur] &&& }\]
Let $\PP$ be the totally acyclic complex
\[ \xymatrix{\PP: & \cdots \ar[r] & X^{m-2} \ar[r] & X^{m-1} \ar[rr]\ar[dr]  && Q^0 \ar[r] & Q^1 \ar[r] & \cdots \\
&&&& \Ker \partial_{\X}^m \ar[ur] &&& }\]
of finitely generated projective $A$-modules.
We have the following triangle in $\K(\prj A)$
\[ \X {}_{{}_{m-1}\sqsupset}[-m] \rt \PP \rt {}_{\sqsubset_0} \BQ \rightsquigarrow.\]
Now, apply functor $\Phi$ on the above triangle to obtain a triangle
\[(**) \ \  \Phi(\X {}_{{}_{m-1}\sqsupset}[-m]) \rt \Phi(\PP) \rt \Phi({}_{\sqsubset_0 }\BQ) \rightsquigarrow,\]
in $\K(\Prj B)$, see Theorem \ref{Thmder}.
Since $\PP \in \K_{\tac}(\prj A)$, Theorem \ref{Thmder} implies  that $\Phi(\PP)$ lies in $\K_{\tac}(\prj B)$. Also, by \cite[Proposition 7.12]{N08}, ${}_{\sqsubset_0 }\BQ$ is a compact object of $\K(\Prj A)$ and so $\Phi({}_{\sqsubset_0 }\BQ)$ will be a compact object of $\K(\Prj B)$. Hence $\Phi({}_{\sqsubset_0} \BQ)$ lies in $\K^+(\prj B)$.

Therefore, it follows from triangle $(**)$ that $\Phi(\X {}_{{}_{m-1}\sqsupset[-m]})$ belongs to $\K^{-, \bb}_t(\prj B)$. Consequently, triangle $(*)$ implies that $\Phi(\X)$ belongs to $\K^{-, \bb}_t(\prj B)$.

Let $\Psi$ be the quasi-inverse of the equivalence
\[\Phi : \K(\Prj A) \lrt \K(\Prj B).\]
Then the same argument as above can be applied to prove that $\Psi$ sends any complex in $\K^{-, \bb}_{\rm t}(\prj B)$ to a complex in $\K^{-, \bb}_{\rm t}(\prj A)$, up to isomorphism. So, we have an equivalence $\Phi| : \K^{-, \bb}_{\rm t}(\prj A) \lrt \K^{-, \bb}_{\rm t}(\prj B)$. The proof is hence complete.
\end{proof}

For a noetherian ring $A$, it is known that there is a full and faithful functor $H: \underline{\Gprj} A \lrt \D_{\sg}^{\bb}(A)$. Moreover, if $A$ is a Gorenstein ring, then $H$ is an equivalence, see \cite{Buc, Hap2}. Bergh, J{\o}rgensen and Oppermann \cite{BJO}, introduced the notion of the Gorenstein defect category as the verdier quotient  $\D_{\rm defect}^{\bb}(A):= \D^{\bb}_{\sg}(A)/ \im H$. The Gorenstein defect category measures `how far' $A$ is from being Gorenstein in the sense that $\D^{\bb}_{\rm defect}(A)=0$ if and only if $A$ is Gorenstein.

Kong and Zhang \cite[Theorem 6.8]{KZ} give the following  description of the Gorenstein defect category, when $A$ is a coherent ring
  $$\D_{\rm defect}^{\bb}(A) \simeq \K^{-, \bb}(\prj A)/\K^{-,\bb}_t(\prj A).$$

As a direct consequence of Proposition  \ref{Stder}, we have the following corollary.

\begin{scorollary}
Let $A$ and $B$ be two virtually Gorenstein algebras that are derived equivalent. Then there is an equivalence
$$\D^{\bb}_{\rm defect}( A) \simeq \D^{\bb}_{\rm defect}( B)$$
of triangulated categories.
\end{scorollary}

\begin{proof}
In view of Theorem \ref{KelThm} and Proposition \ref{Stder}, the equivalence $\Phi: \K^-(\Prj A) \lrt \K^-(\Prj B)$ induces equivalences
$$ \K^{-, \bb}(\prj A) \simeq \K^{-, \bb}(\prj B) \ \ \text{and} \ \ \K^{-, \bb}_{\rm t}(\prj A) \simeq \K^{-, \bb}_{\rm t}(\prj B)$$
of  triangulated categories.
Thus we have an equivalence
$$ \K^{-, \bb}(\prj A)/ \K^{-, \bb}_{\rm t}(\prj A) \simeq \K^{-, \bb}(\prj B)/ \K^{-, \bb}_{\rm t}(\prj B)$$
of triangulated quotient categories. That is
$$ \D^{\bb}_{\rm defect}(A) \simeq \D^{\bb}_{\rm defect}(B).$$
\end{proof}

Let $F: \D^{\bb}(\mmod A) \lrt \D^{\bb}(\mmod B)$ be a derived equivalence  with the quasi-inverse $G$. Suppose  that $\T \in \K^{\bb}(\prj A)$, resp. $\T' \in \K^{\bb}(\prj B)$,  is the tilting complex associated to $F$, resp. $G$. By \cite[Lemma 2.1]{HX2}, we assume that $\T$  and $\T'$ are  complexes  of the form
\[ \T: \ \ \ 0 \lrt T^{-n} \lrt T^{-n+1} \lrt \cdots \lrt T^{-1} \lrt T^0 \lrt 0,\]
\[ \T': \ \ \ 0 \lrt T'^0 \lrt T'^1 \lrt \cdots \lrt T'^{n-1} \lrt T'^n \lrt 0 .\]
We fix these notations towards the end of this subsection. \\

We need the following lemma, that is quoted from \cite[Lemma 2.2]{HX2}.
\begin{slemma}\label{HX}
Let $\X$ be a bounded above and $\Y$ be a bounded below complex over a ring $A$. If there is an integer $m$, such that $X^i$ is projective for all $i>m$ and $Y^j=0$ for all $j<m$, then $\Hom_{\D(\Mod A)}(\X, \Y) \cong \Hom_{\K(\Mod A)}(\X,\Y)$.
\end{slemma}

\begin{slemma}\label{lem1'}
Let $A$ and $B$ be two  noetherian rings that are ${\rm fGd}$-derived equivalent. Let $F: \D^{\bb}(\mmod A) \lrt \D^{\bb}(\mmod B)$ be an equivalence with the quasi-inverse $G$.
\begin{itemize}

\item[$(i)$] Assume that $X$ is a finitely generated Gorenstein projective $B$-module. Then $G(X)$ is isomorphic in $\D^{\bb}(\mmod A)$ to a complex of the form
\[ {\bf T}_X: \ \ 0 \lrt T_X^{-n} \lrt T_X^{-n+1} \lrt \cdots \lrt T_X^{-1} \lrt T_X^0 \lrt 0\]
where $T_X^{-n}$  is Gorenstein projective and $T_X^i$ is projective,  for $-n+1 \leq i \leq 0.$

\item[$(ii)$] Assume that $X$ is a finitely generated Gorenstein projective $A$-module. Then $F(X)$ is isomorphic in $\D^{\bb}(\mmod B)$ to a complex of the form
\[ {\bf T}'_X: \ \ 0 \lrt {T'}_X^{0} \lrt {T'}_X^{1} \lrt \cdots \lrt {T'}_X^{n-1} \lrt {T'}_X^n \lrt 0\]
where ${T'}_X^0$ is Gorenstein projective and ${T'}_X^i$ is projective, for $1 \leq i \leq n.$
\end{itemize}
\end{slemma}

\begin{proof}
We just prove statement $(i)$. Statement $(ii)$ follows similarly.
In view of Lemma 3.3 of \cite{P}, $G(X)$ is isomorphic in $\D^{\bb}(\mmod A)$ to a complex of the form
\[ 0 \lrt T_X^{-n} \lrt T_X^{-n+1} \lrt \cdots \lrt T_X^0 \lrt 0\]
where $T_X^{-n} \in {}^\perp A$ and for $-n+1 \leq i \leq 0$, $T_X^i \in \prj A$.

On the other hand, the commutative diagram
\[\xymatrix{\D^{\bb}_{\rm fGd}(\mmod B) \ar@{^{(}->}[r] \ar[d]^{G|} & \D^{\bb}(\mmod B) \ar[d]^{G} \\ \D^{\bb}_{\rm fGd}(\mmod A) \ar@{^{(}->}[r]  & \D^{\bb}(\mmod A) } \]
and Proposition 2.10 of \cite{Ka} imply that $G(X)$ is quasi-isomorphic to a bounded complex of finitely generated Gorenstein projective $A$-modules.

Take a K-projective resolution $\PP$ of $G(X)$. By induction on the length of $G(X)$, one can see that $\PP$ can be chosen in $\K^{-, \bb}_{\rm t}(\prj A)$. So there is an integer $m=m(\PP)$ such that ${\rm H}^i(\PP)=0$ and $\Ker \partial_{\PP}^i \in \Gprj A$ for all $i\leq m$. Consider the images of $G(X)$ and $\PP$  in $\D_{\rm sg}^{\bb}( A)$ under the quotient functor $Q: \D^{\bb}(\mmod A) \lrt \D^{\bb}_{\sg}(A)$. Then there is an isomorphism $T_X^{-n}[n] \cong \Ker \partial_{\PP}^m[m]$ in $\D_{\sg}^{\bb}(A)$. If we set $Y:= \Ker \partial_{\PP}^m[m]$,
there exists a roof $fs^{-1} : Y \lrt T_X^{-n}[n]$ in which both $f$ and $s$ can be completed into triangles in $D^{\bb}(\mmod A)$ such that their cones belong to $\K^{\bb}(\prj A)$. More precisely, there is a complex ${\bf Z} \in \D^{\bb}(\mmod A)$ and triangles
\[ {\bf Z} \st{s}\rt Y \rt \W \rightsquigarrow,\]
\[{\bf Z} \st{f}\rt T_X^{-n}[n] \rt \W' \rightsquigarrow,\]
such that $\W$ and $\W'$ belong to $\K^{\bb}(\prj A)$. Let $\PP_Y$, resp. ${\bf Q}$, $\PP_{\bf Z}$, be a K-projective resolution of $Y$, resp. $T_X^{-n}[n]$, ${\bf Z}$. So the above triangles can be written of the form
\[ \PP_{\bf Z} \st{s}\rt \PP_Y \rt \W \rightsquigarrow,\]
\[\PP_{\bf Z} \st{f}\rt {\bf Q} \rt \W' \rightsquigarrow.\]
Since $\W \in \K^{\bb}(\prj A)$, the first triangle implies that there is an integer $d$, such that $\Omega^j \PP_{\bf Z} \in \Gprj A$, for $j<d$.

Now, the same argument, applying this time to ${\bf Q}$ in the second triangle, implies that $T_X^{-n}$ has finite Gorenstein projective dimension. Therefore, by \cite[Theorem 2.10]{Ho} there is a short exact sequence
$0 \lrt L \lrt G \lrt T_X^{-n} \lrt 0$, where $G \in \Gprj A$ and $L$ has finite projective dimension. The fact that $T_X^{-n} \in {}^\perp A$, implies that this short exact sequence splits and so $T_X^{-n}$ lies in $\Gprj A$.
\end{proof}

\begin{slemma}\label{lem2'}
  Using the same notations as in  Lemma \ref{lem1'}, for each pair $X, Y \in \Gprj B$ and $i \neq 0$, we have $\Hom_{\K^{\bb}(\mmod A)}({\bf T}_X, {\bf T}_Y[i])=0$.
\end{slemma}

\begin{proof}
The result can be obtained from a simple modification of the proof of \cite[Lemma 3.7]{P}. For the convenience of the reader, we include the sketch of the proof here.

Lemma \ref{HX} implies that $\Hom_{\K^{\bb}(\mmod A)} (\T_X, \T_Y [i])=0$, for $i<0$.
For $i>0$, consider the following triangles in $\K^{\bb}(\mmod A)$
\[ \bar{\T}_X \lrt \T_X \lrt T_X^{-n}[n] \lrt \bar{\T}_X[1], \]
\[ \bar{\T}_Y \lrt \T_Y \lrt T_Y^{-n}[n] \lrt \bar{\T}_Y[1],\]
where $\bar{\T}_X$, resp.  $\bar{\T}_Y$, denotes the complex ${}_{\sqsubset_{-n+1}} \T_X$, resp. ${}_{\sqsubset_{-n+1}} \T_Y$.

The second triangle implies that ${\rm H}^i(F(\bar{\T}_Y))=0$, for all $i>1$. So $F(\bar{\T}_Y)$ is isomorphic in $\D^{\bb}(\mmod B)$ to a complex $\BQ$ of the form
\[ \cdots \lrt  Q^{-1} \lrt Q^0 \lrt Q^1 \lrt 0 .\]

Now, apply the cohomological functors $\Hom_{\K(\mmod A)}(-, \bar{\T}_Y[i])$ and $\Hom_{\D(\mmod A)}(- , \bar{\T}_Y[i])$ on the first triangle and the homological functor $\Hom_{\K(\mmod A)}(\T_X , -)$ on the second one. One can deduce that $\Hom_{\K(\mmod A)} (\T_X, \T_Y[i])=0$ for $i>1$.

To prove that $\Hom_{\K^{\bb}(\mmod A)}(T_X, T_Y[1])=0$, it is enough to show that the induced map
\[ \Hom_{\K^{\bb}(\mmod A)}(\T_X, T_Y^{-n}[n]) \lrt \Hom_{\K^{\bb}(\mmod A)}(\T_X, \bar{\T}_Y[1])\]
is surjective. Consider the commutative diagram
\[\xymatrix@C-0.5pc@R-0.5pc{ \BQ \ar[r]^{\al} \ar[d]^{\cong} & Y \ar[r] \ar[d]^{\cong} & {\rm cone}(\al) \ar[r] \ar[d]^{\cong} & \BQ[1] \ar[d]^{\cong}\\
F(\bar{\T}_Y ) \ar[r] & F(\T_Y) \ar[r] & F(T_Y^{-n}[n]) \ar[r] & F(\bar{\T}_Y[1]).}\]
Apply  Lemma \ref{HX} to get the following commutative diagram whose vertical maps are isomorphisms
\[\xymatrix@C-0.5pc@R-0.5pc{ \Hom_{\K^{\bb}(\mmod A)}(\T_X, T_Y^{-n}[n]) \ar[r] \ar[d]^{\cong} & \Hom_{\K^{\bb}(\mmod A)}( \T_X, \bar{\T}_Y[1]) \ar[d]^{\cong}\\
\Hom_{\D^{\bb}(\mmod B)}(X, {\rm cone}(\al)) \ar[r] & \Hom_{\D^{\bb}(\mmod B)}(X, \BQ[1]).}\]

So, it is enough to prove the following isomorphisms
\[ \Hom_{\D^{\bb}(\mmod B)} (X, \BQ[1]) \cong \Hom_{\K^{\bb}(\mmod B)}(X, \BQ[1]), \] \[\Hom_{\D^{\bb}(\mmod B)}(X, {\rm Cone}(\al)) \cong \Hom_{\K^{\bb}(\mmod B)}(X, {\rm Cone}(\al)).\]
The first isomorphism can be proved by induction on the length of $\BQ$, and the second one is obtained by applying the homological functors $\Hom_{\D^{\bb}(\mmod B)}(X, -)$ and $\Hom_{\K^{\bb}(\mmod B)}(X,-)$ on the triangle
$$ {\rm cone}(\al)^{0} \lrt {\rm cone}(\al) \lrt {\rm cone}(\al) {}_{{}_{-1} \sqsupset} \lrt {\rm cone}(\al)^{0} [1] $$
in $\K^{\bb}(\mmod B)$. Note that in the above triangle, ${\rm cone}(\al)^{0}$ denotes the stalk complex with the zeroth term of the complex ${\rm cone}(\al)$ in degree zero.
\end{proof}

\begin{slemma}\label{lem3'}
Let $A$ be a ring. Then there exists an equivalence
$$ \K^{\bb}(\Gprj A) \simeq \K^{\bb}(\prj (\Gprj A))$$
of triangulated categories.
\end{slemma}

\begin{proof}
Consider the functor
$$\varphi: \Gprj A \lrt \prj \FMod(\Gprj A)$$
given by $\varphi(X)=\Hom_A(-, X)$. Every object of $\prj \FMod (\Gprj A)$ is of the form $\Hom_A(-,X)$ for some $X \in \Gprj A$. So, by using Yoneda lemma one can easily deduce that $\varphi$ is an equivalence. This equivalence can be extended, in a natural way, to the desired equivalence
$$\bar{\varphi}: \K^{\bb}(\Gprj A) \st{\sim}\lrt \K^{\bb}(\prj (\Gprj A))$$
of triangulated categories.
\end{proof}

\begin{stheorem}\label{GorversDer}
Let $A$ and $B$ be two  noetherian rings that are ${\rm fGd}$-derived equivalent. Then there is an equivalence $\D(\Mod \Gprj A) \simeq \D (\Mod\Gprj B)$ of triangulated categories.
\end{stheorem}

\begin{proof}
Let $F: \D^{\bb}(\mmod A) \lrt \D^{\bb}(\mmod B)$ be the derived equivalence with the quasi-inverse $G$. Let  $\mathbb{E}$ be the set of all $ G(X)$, where $X$ runs through isomorphism classes of finitely generated Gorenstein projective $B$-modules.  We claim that $\mathbb{E}$ has the following properties:
\begin{itemize}
\item[$(i)$] $\Hom_{\K^{\bb}(\Gprj A)}(E, E' [i])=0$, for all $E, E' \in \mathbb{E}$ and $i \neq 0$.
\item[$(ii)$] $\mathbb{E}$ generates $\K^{\bb}(\Gprj A)$ as a triangulated category.
\item[$(iii)$] $\mathbb{E}$ is equivalent to $\Gprj B$.
\end{itemize}
By Lemma \ref{lem1'}, $\mathbb{E}$ is a subcategory of $\K^{\bb}(\Gprj A)$. Also, Lemma \ref{lem2'} implies that condition $(i)$ holds true. Moreover, two categories $\mathbb{E}$ and $\Gprj B$ are equivalent via $G$. Indeed, let $X$ and $Y$ be modules in $\Gprj B$. Then there exist the following isomorphisms
\[\begin{array}{ll}
\Hom_B(X, Y) & \cong \Hom_{\D(\mmod A)}(G(X), G(Y))\\
& \cong \Hom_{\K(\mmod A)}(G(X), G(Y)).
\end{array}\]
Note that the last isomorphism follows from Lemma \ref{HX}. This means that the functor $G: \Gprj B \lrt \mathbb{E}$ is full and faithful. Also, clearly $G$ is dense and hence is an equivalence.

Therefore, we just need to prove that $\mathbb{E}$ generates $\K^{\bb}(\Gprj A)$ as a triangulated category.

First note that, since $G$ is a derived equivalence,  $\add G(B)$ generates $\K^{\bb}(\prj A)$ as a triangulated category. Now, assume that $Y$ is a finitely generated Gorenstein projective $A$-module that is not projective. By Lemma \ref{lem1'}, $F(Y)$ is isomorphic in $\D^{\bb}(\mmod B)$ to a complex
$$ 0 \lrt T'^0_Y \lrt \cdots \lrt T'^n_Y \lrt 0$$
with $T'^0_Y$ Gorenstein projective and $T'^i_Y$ projective for every $1 \leq i \leq n$.
Consider the triangle ${}_{\sqsubset_1} F(Y) \rt F(Y) \rt T'^0_Y \rightsquigarrow$ and apply the functor $G$ on it, to get the triangle
$$ (*) \ \ G({}_{\sqsubset_1} F(Y)) \rt GF(Y) \rt G(T'^0_Y) \rightsquigarrow.$$
Since ${}_{\sqsubset_1} F(Y)$ is a bounded complex of finitely generated projective $A$-modules, $G({}_{\sqsubset_1} F(Y))$ is also a bounded complex of finitely generated projective $B$-modules. Moreover, by Lemma \ref{lem1'}, $G(T'^0_Y)$ is isomorphic in $\D^{\bb}(\mmod A)$ to a complex
$$0 \lrt G^{-n} \lrt P^{-n+1} \lrt \cdots \lrt P^0 \lrt 0$$
with $G^{-n} \in \Gprj A$ and $P^i \in \prj A$ for all $-n+1 \leq i \leq 0$.
Now, consider the image of triangle $(*)$ in $\D^{\bb}_{\sg}(A)$. We have an isomorphism $ G(T'^0_Y)\cong G^{-n}[n] \cong Y = GF(Y)$ in $\D^{\bb}_{\sg}(A)$.

On the other hand, since $Y$ is Gorenstein projective, there is a complete projective resolution
\[ \xymatrix{ \cdots \ar[r] &Q^{-n} \ar[r]^{\partial^{-n}} &  \cdots \ar[r] & Q^{-1} \ar[rr]^{\partial^{-1}} \ar[rd] && Q^0 \ar[r] & \cdots \ar[r] & Q^n \ar[r]^{\partial^n} & \cdots \\ &&& & Y \ar[ru] &&&}\]
of $Y$ with $Q^i \in \prj A$. Set $Y' := \Ker \partial^n$. Since $Y' \in \Gprj A$, one may apply the same argument as above, to $Y'$, to get  that there exists an $A$-module $X' \in \Gprj B$ such that $G(X') \cong Y'$ in $\D_{\sg}^{\bb}(A)$. Moreover, $G(X')$ is isomorphic in $D_{\sg}^{\bb}(A)$ to a complex $G[n]$ for some finitely generated Gorenstein projective $A$-module $G$. Also, there is a triangle
$$ Q \rt G(X') \rt G[n] \rightsquigarrow$$
in $\K^{\bb}(\Gprj A)$, where $Q$ is a bounded complex of projectives. This implies that $G[n]$ belongs to ${\rm thick}(\BE)$.

Consider the triangle
$$Y \rt \BQ' \rt Y'[-n] \rightsquigarrow,$$
where $\BQ'$ is the truncated complex
$$0 \lrt Q^0 \lrt \cdots \lrt Q^{n-1} \lrt 0.$$
Thus, $Y' \cong Y[n]$ in $\D^{\bb}_{\sg}(A)$.
Therefore,  there is an isomorphism $G[n] \cong Y[n]$, and then $G \cong Y$, in $\D_{\sg}^{\bb}(A)$. The full and faithful functor $H: \underline{\Gprj}A \lrt \D_{\sg}^{\bb}(A)$ implies that $G \cong Y$ in $\underline{\Gprj}A$. So there exist projective modules $P$ and $Q$ such that $G \oplus P \cong Y \oplus Q$ in $\Gprj A$. Now, since $G$, $P$ and $Q$ all belong to ${\rm thick} (\mathbb{E})$, $Y$ also belongs to ${\rm thick}(\mathbb{E})$.  This finishes the proof of the claim.

By Lemma \ref{lem3'}, there is an equivalence $\bar{\varphi}: \K^{\bb}(\Gprj A) \lrt \K^{\bb}(\prj (\Gprj A))$. Let $\CT$ be the subcategory of $\K^{\bb}(\prj (\Gprj A))$ that  corresponds to $\mathbb{E}$ via $\bar{\varphi}$. The properties of $\mathbb{E}$, mentioned above, imply that $\CT$ satisfies all conditions $({\rm P}1)$, $({\rm P}2)$ and $({\rm P}3)$ of \ref{Properties}. Moreover, property $(ii)$ of $\mathbb{E}$ implies that a complex $\X$ in $\K(\Mod \Gprj A)$ is acyclic if and only if it is $\CT$-acyclic. So $\D_{\CT}(\Mod (\Gprj A)) \simeq \D(\Mod (\Gprj A))$.

Now Theorem \ref{KelThm} comes to play to give  an equivalence
$\D(\Mod \CT)\simeq \D(\Mod (\Gprj A))$ of triangulated categories. Hence, by property $(iii)$ of $\mathbb{E}$, we get that $\D(\Mod (\Gprj B)) \simeq \D(\Mod (\Gprj A))$ as triangulated categories. The proof is hence complete.
\end{proof}

\sss\label{AHV} Let $R$ be a ring. By \cite[Theorem 3.3]{AHV}, for a ring $R$, we have an equivalence
$$ \D_{\Gp}^{\bb}(\Mod R) \simeq \K^{-, \Gp \bb} (\Add (\Gprj R)),$$
where $\K^{-, \Gp \bb} (\Add (\Gprj R))$ is the subcategory of $\K^{-} (\Add (\Gprj R))$ consisting of all complexes  $\X$ in which there exists an integer $n=n(\X)$ such that $\HT^i(\Hom(G,{\bf X}))=0$ for all $i \leq n$ and all $G\in \Gprj R$.

\begin{scorollary}\label{DerverGder}
Let $A$ and $B$ be noetherian rings that are ${\rm fGd}$-derived equivalent. Then they are Gorenstein derived equivalent, i.e.
$$ \D^{\bb}_{\Gp}(\Mod A) \simeq \D^{\bb}_{\Gp}(\Mod B).$$
\end{scorollary}
\begin{proof}
By Theorem \ref{GorversDer}, there is an equivalence
$$ \D^{\bb}(\FMod(\Gprj A)) \st{\sim} \lrt \D^{\bb}(\FMod(\Gprj B))$$
of triangulated categories.

Also, the same argument as in the proof of Lemma \ref{lem3'} works to prove equivalences
$$ \K^{-, \Gp \bb}(\Add (\Gprj A)) \simeq \K^{-, \bb}(\Prj (\Gprj A)) \ \text{and} \ \K^{-, \Gp \bb}(\Add (\Gprj B)) \simeq \K^{-, \bb}(\Prj (\Gprj B))$$
of triangulated categories.
Now, the result follows from \ref{AHV} and the known triangulated  equivalences $\D^{\bb}(\FMod (\Gprj A)) \simeq \K^{-,\bb}(\Prj (\Gprj A))$ and $\D^{\bb}(\FMod (\Gprj B)) \simeq \K^{-,\bb}(\Prj (\Gprj B))$
\end{proof}

\begin{scorollary}
Let $\La$ and $\La'$ be finite dimensional algebras that are derived equivalent. Then there exists an equivalence
$$ \D^{\bb}_{\Gp}(\Mod \La) \simeq \D^{\bb}_{\Gp}(\Mod \La')$$
of triangulated categories.
\end{scorollary}

\begin{proof}
 Theorem 4.2 of \cite{Ka} guarantees that $\La$ and $\La'$ are ${\rm fGd}$-derived equivalence. Now, Corollary \ref{DerverGder} implies the result.
\end{proof}

\subsection{Infinitely generated $n$-tilting modules}
Infinitely generated tilting modules arise naturally to extend some aspects of the classical tilting theory to infinitely generated modules over arbitrary rings. The first instance of a generalization of Brenner and Butler's theorem for infinitely generated tilting modules was  given by Facchini  \cite{F1} and \cite{F2}.

Recently, infinitely generated tilting modules over arbitrary rings has received considerable attention towards knowing derived categories and equivalences of general rings, see \cite{B, BMT, NS}. Also, Chen and Xi \cite{CX} studied a relation between two concepts of good tilting modules and recollements of derived categories of rings.

Let $A$ be a ring. A (right) $A$-module $T$ (possibly infinitely generated) is called $n$-tilting if the following conditions hold true
\begin{itemize}
\item[$(i)$] ${\rm pd}_A T\leq n$;
\item[$(ii)$] for all $i>0$ and every cardinal $\al$, $\Ext_A^i(T, T^{(\al)})=0$;
\item[$(iii)$] there exists an exact sequence
$$ 0 \lrt A \lrt T^0 \lrt T^1 \lrt \cdots \lrt T^m \lrt 0,$$
with $T^i \in \Add T$ for $0 \leq i \leq m$.
\end{itemize}

An $n$-tilting module $T$ is called a good $n$-tilting, if there exists an exact sequence
$$ 0 \lrt A \lrt T^0 \lrt T^1 \lrt \cdots \lrt T^n \lrt 0,$$
where for every $i$, $T^i$ is isomorphic to a direct summand of finite direct sums of copies of $T$.

Bazzoni et al. \cite{BMT} proved the following important results on good $n$-titling modules.

\begin{sproposition}\label{1.4}\cite[Proposition 1.4]{BMT}
Let $T_A$ be a good $n$-tilting module and $B= \End_A(T)^{\op}$. Then
\begin{itemize}
\item[$(i)$] $ T_B$ admits a projective resolution
$$ 0 \lrt Q_n \lrt \cdots \lrt Q_0 \lrt  T_B \lrt 0,$$
in which for every $0 \leq i \leq n$, $Q_i$ is finitely generated projective  $B$-module;
\item[$(ii)$] $\Ext^i_B(T,T)=0$ for all $i >0$;
\item[$(iii)$] there exists a ring isomorphism $\End_B(T) \cong A$.
\end{itemize}
\end{sproposition}

We use this result to prove that every good $n$-tilting module $T_A$ provide an equivalence between the derived category $\D(\Mod A)$ and the relative derived category $\D_{T}( \Mod \End_A(T)^{\op})$.

\begin{stheorem}\label{inftilt}
Assume that $T$ is a good $n$-tilting module and $B= \End_A(T)^{\op}$. Then there exists an equivalence
$$ \D(\Mod A ) \simeq \D_{T}(\Mod B),$$
where $\D_{T}(\Mod B) := \D_{\add T}(\Mod B)$.
\end{stheorem}

\begin{proof}
In view of Proposition \ref{1.4}, there exists an exact sequence
$$0 \lrt Q_n \lrt \cdots \lrt Q_0 \lrt T \lrt 0,$$
with $Q_i$ finitely generated projective $B$-modules,for every $0 \leq i \leq n$.
Let $\BT$ be the truncated complex
$$ 0 \lrt Q_n \lrt \cdots \lrt Q_0 \lrt 0$$
of left finitely generated $B$-modules.

It can be easily checked, using Proposition \ref{1.4}, that $\CT= \{ \BT \}$ satisfies Conditions $({\rm P}1)$, $({\rm P}2)$ and $({\rm P}3)$ of \ref{Properties}. So there is a triangulated equivalence
$$ \Phi: \K \K_{\FT}(\Mod B) \lrt \K_{\CT}(\Mod B).$$

Now, observe that \ $\K\K_{\FT}(B \mbox{-} {\rm Mod})$ is equivalent to \ $\K_{\rm prj}(\Mod \End({\bf T}))$ via the functor $\Hom_{\K\K(\Mod B)}(\BT, -)$.
Also, Lemma \ref{lem6} implies that
$\K_{\CT}(\Mod B ) \simeq \D_{\CT}(\Mod B).$
Therefore, we have an equivalence
$$ \D(\Mod \End_{B}(\BT)) \st{\sim} \lrt \D_{\CT}(\Mod B)$$
of triangulated categories.

Since $\End(\BT) \cong \End_B(T)$ in $\D(\Mod B)$ and $\End_B(T) \cong A$ by Proposition \ref{1.4}, there is the desired equivalence
$$\D(\Mod A) \simeq \D_{T}(\Mod B)$$
of triangulated categories.
\end{proof}

\section{Existence of recollements}
Our aim in this section is use techniques of the previous section to provide sufficient conditions for the existence of recollements of derived categories of functor categories.

Let $\CA$ be an abelian category and $\Prj \CA$ be the class of projective objects in $\CA$. A complex $\T \in \K^{\bb}(\Prj \CA)$ is called a partial tilting complex if it satisfies the following two properties.
\begin{itemize}
\item[$(i)$] $\Hom_{\K^{\bb}(\Prj \CA)}(\T, \T[i])=0$, for all $i \neq 0$.
\item[$(ii)$] If $\{\T_i\}_{i \in I}$ is an index family of copies of $\T$, then
$$  \Hom_{\K^{\bb}(\Prj \CA)}(\T, \oplus_{i \in I}\T_i) \cong \oplus_{i\in I} \Hom_{\K^{\bb}(\Prj \CA)}(\T, \T_i). $$
\end{itemize}

For any $\X \in \D^-(\CA)$, $\X ^\perp$ is a full triangulated subcategory of $\D^-(\CA)$ formed by all complexes $\Y \in \D^-(\CA)$ in which $\Hom_{\D^-(\CA)}(\X, \Y[i])=0$, for all $i \in \Z$.
Let $\CT$ be a subcategory of $\D^-(\CA)$. We set $\CT^\perp = \bigcap_{\X \in \CT} \X^\perp$.

Let $A$, $B$ and $C$ be rings that are associative with identity. Koenig proved that derived category $\D^-(\Mod A)$ of $A$ admits a recollement
\vspace{0.3cm}
\[\xymatrix@C=0.5cm{ \D^-(\Mod B) \ar[rrr]  &&& \D^-(\Mod A) \ar[rrr] \ar@/^1.5pc/[lll] \ar@/_1.5pc/[lll] &&& \D^-(\Mod C)\ar@/^1.5pc/[lll] \ar@/_1.5pc/[lll] }\]
\vspace{0.3cm}
\\relative to $\D^-(\Mod B)$ and $\D^-(\Mod C)$ if and only if there exist partial tilting complexes $\FB \in \K^{\bb}(\Prj A)$ and $\FC \in \K^{\bb}(\prj A)$ satisfying the following properties
\begin{itemize}
\item[$(i)$] $\End_A(\FB)\cong B$,
\item[$(ii)$] $\End_A(\FC)\cong C$,
\item[$(iii)$] $\Hom_{\D^-(\Mod A)}(\FC, \FB[i])=0$, for all $i\in \Z$,
\item[$(iv)$] $\FB^\perp \cap \FC^\perp ={0}$.
\end{itemize}

In the following, we obtain a sufficient conditions for the existence of recollements of functor categories. It also provide a different proof for the one direction of Koenig's result stated above.

\begin{theorem}\label{recollement}
Let $\CS$ be a small category. Let $\FB$, resp. $\FC$, be a full triangulated small subcategory of $\K^{\bb}(\Prj \CS)$, resp. $\K^{\bb}(\prj \CS)$, satisfying  conditions $({\rm P}1)$, $({\rm P}2)$ and $({\rm P}3)$, defined in \ref{Properties}.
Assume further that
\begin{itemize}
\item[$({\rm P}4)$] $\Hom_{\D(\Mod \CS)}(T , T'[i])=0$, for every $T$ in $\FC$ and $T'$ in $\FB$, and
 \item[$({\rm P}5)$] $\FB^\perp \bigcap\FC^\perp=0$.
 \end{itemize}
 Then there exists the following recollement
 \vspace{0.3cm}
 \[\xymatrix@C=0.5cm{ \D(\Mod \FB) \ar[rrr]  &&& \D(\Mod \CS) \ar[rrr] \ar@/^1.5pc/[lll] \ar@/_1.5pc/[lll] &&& \D(\Mod \FC)\ar@/^1.5pc/[lll] \ar@/_1.5pc/[lll] }\]
 \vspace{0.2cm}
\end{theorem}

\begin{proof}
By Proposition \ref{Phi-equi}, there are the following triangulated functors
$$ \Phi_{\FB} : \K\K_{\mathfrak X}(\Mod \CS) \lrt \K(\Mod \CS),$$
$$ \Phi_{\FC}: \K\K_{\mathfrak Y}(\Mod \CS) \lrt \K(\Mod \CS),$$
where ${\mathfrak X}$, resp.\ ${\mathfrak Y}$, denotes the class of objects of $\K\K(\Mod \CS)$ formed by all shifting of stalk complexes ${\bf B}$, resp.\ ${\bf C}$, with $B$, resp.\ $C$, in degree zero, for every object $B\in \FB$, resp.\ $C \in \FC$.
Set $\CU = \im \Phi_{\FC}$, $\CV= \im \Phi_{\FB}$ and $\CW= (\im \Phi_{\FB})^\perp$.
We show that $(\CU,\CV)$ and $(\CV, \CW)$ are stable $t$-structures in $\D(\Mod \CS)$. In view of Proposition \ref{Phi-equi},  $\im \Phi_{\FC}= {\rm Loc} \FC$ and $ \im \Phi_{\FB}= {\rm Loc} \FB$.
Since $\FC \subset \K^{\bb}(\prj \CS)$, every object of $\FC$ is compact. By property $({\rm P}4)$, $\Hom_{\D(\Mod \CS)}(C, B[i])=0$, for every $C \in \FC$ and every $B \in \FB$. Now, the argument as in the proof of Lemma \ref{Xj} shows that $\Hom_{\D(\Mod \CS)}(C, {\bf B}[i])=0$ for every $C \in \FC$ and every ${\bf B} \in {\rm Loc}\FB$. Now, using this argument again, one can deduce that $\Hom_{\D(\Mod \CS)}({\bf C}, {\bf B}[i])=0$, for every ${\bf C} \in {\rm Loc}\FC$, ${\bf B}\in {\rm Loc} {\FB}$ and all $i\in \Z$.

Now, suppose that $\X$ is an object of $\D(\Mod \CS)$. By the proof of Lemma \ref{lem6}, there exists a triangle
$$ \X' \rt \X \rt \X'' \rightsquigarrow,$$
where $\X' \in \CU$ and $\X'' \in \CU^\perp$.
Moreover,  in a similar way, there exists  a triangle
$$ \Y' \rt \X'' \rt \Y'' \rightsquigarrow,$$
in which $\Y' \in \CV$ and $\Y'' \in \CV^\perp$. Assume that $C \in \CV$ and  apply the homological functor $\Hom(C,-)$ to the above triangle, to get that $\Hom(C, \Y''[i])=0$ for all $i \in \Z$. Condition $({\rm P}5)$ yields that $\Y''$ vanishes in $\D(\Mod \CS)$. So $\X''$ is isomorphic to $\Y'$ in $\D(\Mod \CS)$ and then belongs to $\CV$.

Therefore $(\CU, \CV)$ and $(\CV, \CW)$ are  stable $t$-structures in $\D(\Mod \CS)$.
Hence, by Proposition \ref{Miyachi}, we have the following recollement
\[\xymatrix@C=0.5cm{\CV \ar[rrr]^{i_*}  &&& \D(\Mod \CS) \ar[rrr]^{j^*} \ar@/^1.5pc/[lll]_{i^!} \ar@/_1.5pc/[lll]_{i^*} &&& \D(\Mod \CS)/\CV, \ar@/^1.5pc/[lll]_{j_*} \ar@/_1.5pc/[lll]_{j_!} }\]
\vspace{0.3cm}
\\such that $\im j_!=\CU$ and $\im j_*= \CW$. By Remark \ref{rem1}, $\CV$ is equivalent to $\D(\Mod \FB)$ and $\CU$ is equivalent to $\D(\Mod \FC)$. Furthermore, definition of recollement implies that $j_!$ induces an equivalence $\CU \simeq \D(\Mod \CS) / \CV$. Hence, there is an equivalence
$$ \D(\Mod \CS) / \CV \simeq \D(\Mod \FC)$$
of triangulated categories. Consequently, we have the desired recollement.
\end{proof}

\begin{proposition}\label{resRec}
With the assumptions as in the above theorem, we have a recollement
\vspace{0.35cm}
 \[\xymatrix@C=0.5cm{ \D^-(\Mod \FB) \ar[rrr]  &&& \D^-(\Mod \CS) \ar[rrr] \ar@/^1.5pc/[lll] \ar@/_1.5pc/[lll] &&& \D^-(\Mod \FC).\ar@/^1.5pc/[lll] \ar@/_1.5pc/[lll] }\]
 \vspace{0.2cm}
 \\Moreover, the inclusion functor $\iota: \D^-(\Mod \CS) \lrt \D(\Mod \CS)$ is a morphism between the above recollement and the recollement in Theorem \ref{recollement}, i.e. there is the following commutative diagram of recollements
 \vspace{0.35cm}
\[\xymatrix@C=0.5cm@R=0.7cm { \D(\Mod \FB) \ar[rrr] \ar@{<-_)}[dd]  &&& \D(\Mod \CS) \ar[rrr] \ar@{<-_)}[dd] \ar@/^1.5pc/[lll] \ar@/_1.5pc/[lll] &&& \D(\Mod \FC) \ar@/^1.5pc/[lll] \ar@/_1.5pc/[lll] \ar@{<-_)}[dd]  \\ \\ \D^-(\Mod \FB) \ar[rrr]  &&& \D^-(\Mod \CS) \ar[rrr] \ar@/^1.5pc/[lll] \ar@/_1.5pc/[lll] &&& \D^-(\Mod \FC)\ar@/^1.5pc/[lll] \ar@/_1.5pc/[lll]  }\]
\vspace{0.2cm}
\end{proposition}

\begin{proof}
 Consider the following subcategories of $\D^-(\Mod \CS)$
 \begin{itemize}
 \item $\CU= \im \Phi_{\FC} \bigcap \D^-(\Mod \CS)$,
  \item $\CV= \im \Phi_{\FB}\bigcap \D^-(\Mod \CS)$,
  \item $\CW= \im \Phi_{\FB}^\perp \bigcap \D^-(\Mod \CS)$,
 \end{itemize}
where, $\Phi_{\FB}: \K\K_{\mathfrak X}(\Mod \CS) \lrt \K(\Mod \CS)$ and $\Phi_{\FC}: \K\K_{\mathfrak Y}(\Mod \CS) \lrt \K(\Mod \CS)$ are introduced in the proof of above theorem.
First we prove that $(\CU, \CV)$ and $(\CV, \CW)$ are stable $t$-structures in $\D^-(\Mod \CS)$. In view of the proof of above theorem, $\Hom_{\D^-(\Mod \CS)}(\CU, \CV)=0 = \Hom_{\D^-(\Mod \CS)}(\CV,\CW)$.

Now, let $\X$ be a complex in $\D^-(\Mod \CS)$. By the proof of  Theorem \ref{recollement}, there is a triangle
$$ \X' \rt \X \rt \X'' \rightsquigarrow$$
in which $\X' \in \im \Phi_{\FC}$ and $\X'' \in \im \Phi_{\FB}$. In view of  condition $({\rm P}4)$ we may proved that $\Hom_{\D(\Mod \CS)}(C, \X''[i])=0$, for every $C \in \FC$ and all $i$.
On the other hand, since $\X \in \D^-(\Mod \CS)$, for each complex ${C} \in \FC$ there exists an integer \ $N$ such that $\ \Hom_{\D(\Mod \CS)}(C, \X[i])=0$, for  all \ $i\geq N$. Therefore, the above triangle implies  that $$\Hom_{\D(\Mod \CS)}(C, \X'[i])=0,$$ for every $C \in \FC$ and all $i >N$. Since $\X' \in {\rm Loc} \FC$, $\X'$ is isomorphic in $\K_{\FC}(\Mod \CS)$ to a bounded above complex. This implies that $\X'$ lies in $\D^-(\Mod \CS)$.  By the above triangle, $\X''$ lies in $\D^-(\Mod \CS)$, up to isomorphism, as well. Consequently, $(\CU, \CV)$ is a stable $t$-structure.

Moreover, as it is shown in the proof of Theorem \ref{recollement}, $(\im \Phi_{\FB}, \im \Phi_{\FB}^\perp)$ is a stable $t$-structure. So we have a triangle
$$\Y \rt \X \rt \BZ \rightsquigarrow,$$
in $\D(\Mod \CS)$ with $\Y \in \im \Phi_{\FB}$ and $\BZ \in \im \Phi_{\FB}^\perp$. Now, the same argument as above works to show that $\Y$, and hence $\BZ$, lies in $\D^-(\Mod \CS)$, up to isomorphism. So $(\CV, \CW)$ is a stable $t$-structure.

Now, Proposition \ref{Miyachi} implies the following recollement
\[\xymatrix@C=0.5cm{\CV \ar[rrr]^{i_*}  &&& \D^-(\Mod \CS) \ar[rrr]^{j^*} \ar@/^1.5pc/[lll]_{i^!} \ar@/_1.5pc/[lll]_{i^*} &&& \D^-(\Mod \CS)/\CV.\ar@/^1.5pc/[lll]_{j_*} \ar@/_1.5pc/[lll]_{j_!} }\]
such that $\im j_!=\CU$ and $\im j_*= \CW$.

We claim that $\CV \simeq \D^-(\Mod \FB)$ and $\CU \simeq \D^-(\Mod \FC)$. In fact, in view of the argument above, $\im \Phi_{\FC}\bigcap \D^-(\Mod \CS) =\K^-_{\FC}(\Mod \CS)$ and $\im \Phi_{\FB} \bigcap \D^-(\Mod \CS) = \K^-_{\FB}(\Mod \CS)$. Now, Proposition \ref{restrict} and Remark \ref{rem1} yield the desired equivalences.

Altogether, we get the following recollement
\vspace{0.35cm}
 \[\xymatrix@C=0.5cm{ \D^-(\Mod \FB) \ar[rrr]  &&& \D^-(\Mod \CS) \ar[rrr] \ar@/^1.5pc/[lll] \ar@/_1.5pc/[lll] &&& \D^-(\Mod \FC)\ar@/^1.5pc/[lll] \ar@/_1.5pc/[lll] }\]
 \vspace{0.2cm}

Furthermore, the inclusion functor $\iota: \D^-(\Mod \CS) \lrt \D(\Mod \CS)$, takes stable $t$-structures $(\CU, \CV)$ to $(\im \Phi_{\FC}, \im \Phi_{\FB})$ and $(\CV, \CW)$ to $(\im \Phi_{\FB}, \im \Phi_{\FB}^\perp)$. Now Corollary 1.13 of \cite{IKM2} implies the desired diagram.
\end{proof}

\begin{remark}
As it is mentioned in Remark \ref{rem}, our proofs also work to generalize the above results to a skeletally small category $\CS$.
\end{remark}

The authors of \cite{AKL3} compared recollements of different levels. They proved that $\D^{\bb}(\Mod)$ level recollements of rings  can be lifted to $\D^-(\Mod)$ and $\D(\Mod)$ levels recollements; see \cite[Sec. 4]{AKL3}. In the following corollary we prove this result using our previous corollaries. Moreover, we show that these recollements fit into a commutative diagram of recollements.

\begin{corollary}
Let $A$ be a ring admitting a $\D^{\bb}(\Mod )$ level recollement
\[\xymatrix@C=0.5cm{ \D^{\bb}(\Mod B) \ar[rrr]^{i_*}  &&& \D^{\bb}(\Mod A) \ar[rrr]^{j^*} \ar@/^1.5pc/[lll]_{i^!} \ar@/_1.5pc/[lll]_{i^*} &&& \D^{\bb}(\Mod C).\ar@/^1.5pc/[lll]_{j_*} \ar@/_1.5pc/[lll]_{j_!} }\]
Then there exist the following inclusion morphisms of recollements
\vspace{0.2cm}

\[\xymatrix@C=0.5cm@R=0.7cm{ \D(\Mod B) \ar[rrr] \ar@{<-_)}[dd]  &&& \D(\Mod A) \ar[rrr] \ar@{<-_)}[dd] \ar@/^1.5pc/[lll] \ar@/_1.5pc/[lll] &&& \D(\Mod C) \ar@/^1.5pc/[lll] \ar@/_1.5pc/[lll] \ar@{<-_)}[dd] \\ \\ \D^-(\Mod B)
\ar[rrr]  \ar@{<-_)}[dd] &&&  \D^-(\Mod A) \ar[rrr] \ar@{<-_)}[dd] \ar@/^1.5pc/[lll] \ar@/_1.5pc/[lll] &&&  \D^-(\Mod C) \ar@{<-_)}[dd] \ar@/^1.5pc/[lll] \ar@/_1.5pc/[lll] \\ \\ \D^{\bb}(\Mod B) \ar[rrr]   &&& \D^{\bb}(\Mod A) \ar[rrr]  \ar@/^1.5pc/[lll] \ar@/_1.5pc/[lll] &&& \D^{\bb}(\Mod C) \ar@/^1.5pc/[lll] \ar@/_1.5pc/[lll] }\]
\vspace{0.2cm}
\end{corollary}

\begin{proof}
Denote by $\FB$, resp. $\FC$, the image of $B$, resp. $C$, under the functor $i_*$, resp. $j_!$. In view of Theorem 1 of \cite{Kon}, $\FB \in \K^{\bb}(\Prj A)$ and $\FC \in \K^{\bb}(\prj A)$ are complexes satisfying the following conditions
\begin{itemize}
\item[$(i)$] $\End(\FB) \cong B$;
\item[$(ii)$] $\End(\FC) \cong C$;
\item[$(iii)$] $\Hom_{\K(\Prj A)}(\FC, \FB[i])=0$, for all $i \in \Z$;
\item[$(iv)$] $\FB^{\perp} \bigcap \FC^\perp =\{0\}$.
\end{itemize}
Hence by the proof of Theorem \ref{recollement}, there are two stable $t$-structures $(\im \Phi_{\FC}, \im \Phi_{\FB})$ and $(\im \Phi_{\FB}, (\im \Phi_{\FB})^\perp)$  inducing a recollement
\vspace{0.3cm}
\[\xymatrix@C=0.5cm{ \D(\Mod B) \ar[rrr]  &&& \D(\Mod A) \ar[rrr] \ar@/^1.5pc/[lll] \ar@/_1.5pc/[lll] &&& \D(\Mod C).\ar@/^1.5pc/[lll] \ar@/_1.5pc/[lll] }\]
\vspace{0.2cm}

On the other hand, it follows from \cite{Mi} that assigned to the recollement
\[\xymatrix@C=0.5cm{ \D^{\bb}(\Mod B) \ar[rrr]^{i_*}  &&& \D^{\bb}(\Mod A) \ar[rrr]^{j^*} \ar@/^1.5pc/[lll]_{i^!} \ar@/_1.5pc/[lll]_{i^*} &&& \D^{\bb}(\Mod C).\ar@/^1.5pc/[lll]_{j_*} \ar@/_1.5pc/[lll]_{j_!} }\]
there are stable $t$-structures $(\im j_!, \im i_*)$ and $(\im i_*, \im j_*)$.

In addition, by Proposition \ref{resRec}, we have the following diagram of recollements
\vspace{0.2cm}

\[\xymatrix@C=0.5cm@R=0.7cm{ \D(\Mod \FB) \ar[rrr] \ar@{<-_)}[dd]  &&& \D(\Mod \CS) \ar[rrr] \ar@{<-_)}[dd] \ar@/^1.5pc/[lll] \ar@/_1.5pc/[lll] &&& \D(\Mod \FC) \ar@/^1.5pc/[lll] \ar@/_1.5pc/[lll] \ar@{<-_)}[dd] \\ \\ \D^-(\Mod \FB) \ar[rrr]  &&& \D^(\Mod \CS) \ar[rrr] \ar@/^1.5pc/[lll] \ar@/_1.5pc/[lll] &&& \D^-(\Mod \FC).\ar@/^1.5pc/[lll] \ar@/_1.5pc/[lll]  }\]
\vspace{0.3cm}
\\Note that by the proof of Proposition \ref{resRec}, the second recollement is obtained by stable $t$-structures $${\tiny (\im \Phi_{\FC} \bigcap \D^-(\Mod A), \im \Phi_{\FB} \bigcap \D^-(\Mod A)) \ }\text{and} \ {\tiny (\im \Phi_{\FB} \bigcap \D^-(\Mod A) , \im \Phi_{\FB}^\perp \bigcap \D^-(\Mod A))}.$$

Now, to complete the diagram, it is enough to prove that inclusion functor $i: \D^{\bb}(\Mod A) \lrt \D^-(\Mod A)$ sends stable $t$-structures $(\im j_!, \im i_*)$ to $(\im \Phi_{\FC} \bigcap \D^-(\Mod A), \im \Phi_{\FB} \bigcap \D^-(\Mod A))$ and $(\im i_*, \im j_*)$ to $(\im \Phi_{\FB} \bigcap \D^-(\Mod A) , \im \Phi_{\FB}^\perp \bigcap \D^-(\Mod A))$.

Let $\X$ belong to $\im j_!$. So, there is a complex $\Y \in \D^{\bb}(\Mod C)$ such that $\X = j_!(\Y)$. The stable $t$-structure $(\im \Phi_{\FC}, \im \Phi_{\FB})$ gives us a triangle
$$ \X' \rt \X \rt \X''\rightsquigarrow,$$
in $\D^-(\Mod A)$ in which $\X' \in \im \Phi_{\FC} \bigcap \D^-(\Mod A)$ and $\X'' \in \im \Phi_{\FB} \bigcap \D^-(\Mod A)$. One should apply the cohomological functor $ \Hom(-, i_*(B))$ to see that $\Hom(\X'' [i], i_*(B))=0$, for all $i \in \Z$.  Since $\X'' \in \lan \FB \ran \bigcap \D^-(\Mod A)$, this implies that  $\X''=0$. Thus, $\X$ is isomorphic to $\X'$.

Now, assume that $\X \in \im i_*$. Consider the stable $t$-structure $$ (\im \Phi_{\FB} \bigcap \D^-(\Mod A) , \im \Phi_{\FB}^\perp \bigcap \D^-(\Mod A)).$$ One should use a similar argument as above and condition $(iv)$ to show that $\X$ is isomorphic to a complex in $\im \Phi_{\FB} \bigcap \D^-(\Mod A)$.

Finally, the adjoint pair $(i_*, i^!)$ yields that $\Hom(i_*(B), \im j_*[i])=0$ for every $i \in \Z$. Thus, since $\im \Phi_{\FB} = {\rm Loc} (i_*(B))$, a standard argument implies that $\Hom({\bf B}, \im j_*[i])=0$, for each ${\bf B} \in \im \Phi_{\FB}$ and for all $i \in \Z$. Therefore, $\im j_*$ is contained in $\im \Phi_{\FB}^\perp \bigcap \D^-(\Mod A)$.

Now, Corollary 1.13 of \cite{IKM2} completes the proof.
\end{proof}

\subsection{Recollements of path rings}
Asashiba \cite{As} proved that if $A$ and $B$ are algebras that are derived equivalent, then their path algebras, incidence algebras and monomial algebras are derived equivalent. Motivated by this result, we show that if we have a $\D^-(\Mod)$ level recollement of rings, then there are both $\D^-(\Mod )$ and also $\D(\Mod)$ level recollements of their path rings, incidence rings and monomial rings over any locally finite quiver. This result should be compared with Theorem 4.6 of \cite{AHV2}, that we proved similar result but in $\D^-(\Mod )$ level and for finite acyclic quivers.  We need to recall some preliminaries.

\sss A quiver $\CQ$ is in fact a directed graph. The set of vertices and arrows of $\CQ$ are denoted by   $\CQ_0$ and  $\CQ_1$, respectively.
A quiver $\CQ$ is said to be locally finite, if the set of paths between every two vertices
is finite.

A relation  is an $A$-linear combination $\rho = \Sigma_{i=0}^m a_i \gamma_i,$ where $a_i \in A$ and $\gamma_i$ are paths of $\CQ$ of length at least 2 having the same source and target. If $m=1$, the relation $\rho$  is called a monomial relation. Also, $\rho$ is called a commutativity relation, provided that it is of the form $\gamma_1 -\gamma_2$.  Let ${I_c}$  denote the set of  all commutativity relations of $\CQ$ and ${I_m}$ denotes the set of monomial relations of $\CQ$. By convention, $I_0$ means no relations.

Let $A$ be a ring. The category of all representations of $\CQ$ over $\Mod A$ will be denoted by $\Rep(\CQ, A)$.
Let $I$ be a set of relations of a quiver $\CQ$. We denote by $\Rep(\CQ_I,A)$  the full subcategory of $\Rep(\CQ,A)$ consisting of all representations $\CM$ such that $\CM_{\rho}=0$, for any relation $\rho \in I$.

\sss {\sc Evaluation functor and its adjoint.}\label{EvaAdj}
Let $\CQ$ be a quiver and $I_*$ be a set of relations, with $* \in \{0, c, m\}$. Assigned to any vertex $v \in \CQ_0$, there is an evaluation functor $e^v: \Rep(\CQ_{I_*}, A) \lrt \Mod A$, mapping any representation $\CM$ of $\CQ$ to its module at vertex $v$, denoted by $\CM_v$. It is known that in the above cases, $e^v$ has both a left and a right adjoint, denoted by $e^v_{\la}$ and $e^v_{\rho}$, respectively. For details on the construction of these adjoints, see \cite{EH} for $*=0$, \cite{Mit} for $*=c$ and \cite{Es} for $*=m$. The evaluation functor $e^v: \Rep(\CQ_{I_*}, A) \lrt \Mod A$ and its adjoints can be extended naturally to the homotopy category level
$k^v: \K(\Rep(\CQ_{I_*}, A) ) \lrt \K(\Mod A)$ with adjoints $k^{v}_{\la}$ and $k^v_{\rho}$. For details see \cite{AEHS}.

Let $A$ be a ring and $\CQ$ be a quiver. The path ring $A\CQ$ is defined to be a free $A$-module with basis all paths of $\CQ$.

Now, assume that $I$ is a set of relations of a quiver  $\CQ$. Let $\CS_{\CQ_I}^A$ denote the category, whose objects are vertices of $\CQ$ and for any $v, w \in \CQ_0$, $\Hom_{\CS_{\CQ_I}^A}(v,w) = \oplus_{\CQ(v,w)}A + \CI/ \CI$, where $\CI$ is an ideal of the path ring $A\CQ$, generated by all relations in $I$. It can be easily checked that the category $\Rep(\CQ_I,A)$ is equivalent to the functor category $(\CS_{\CQ_I}^A, \CA b)$, or equivalently, to $\FMod ((\CS_{\CQ_I}^A)^{\op})$.
Observe that under this equivalence every object $\Hom_{\CS_{\CQ_I}^A}(v,-)$ in $(\CS_{\CQ_I^A}, \CA b)$ is assigned to an object $e_{\la}^v(A)$ in $\Rep(\CQ_I,A)$.

With this consideration, we have the following result as a consequence of the above theorem.

\begin{stheorem}\label{Ext}
Let $\CQ$ be a locally finite quiver with relation $I_*$, where $* \in \{ 0, c, m \}$.  If a ring  $A$ has a $\D^-(\Mod)$ level recollement as follows
\[\xymatrix@C=0.5cm{ \D^-(\Mod B) \ar[rrr]^{i_*}  &&& \D^-(\Mod A) \ar[rrr]^{j^*} \ar@/^1.5pc/[lll]_{i^!} \ar@/_1.5pc/[lll]_{i^*} &&& \D^-(\Mod C),\ar@/^1.5pc/[lll]_{j_*} \ar@/_1.5pc/[lll]_{j_!} }\]
then there exists the following recollement of path rings
\vspace{0.3cm}
\[\xymatrix@C=0.5cm{ \D(\Rep(\CQ_{I_*},B)) \ar[rrr]  &&& \D(\Rep(\CQ_{I_*},A)) \ar[rrr] \ar@/^1.5pc/[lll] \ar@/_1.5pc/[lll] &&& \D(\Rep(\CQ_{I_*},C)).\ar@/^1.5pc/[lll] \ar@/_1.5pc/[lll] }\]
\vspace{0.2cm}
\end{stheorem}

\begin{proof}
For every $v \in \CQ_0$, set $\FB_v= k_{\la}^v(i_*(B))$ and $\FC_v= k_{\la}^v(j_!(C))$. Now, let $\FB= \{ \FB_v \mid v \in \CQ_0\}$ and $\FC =\{ \FC_v \mid v\in \CQ_0\}$.  Using Theorem 1 of \cite{Kon} and the adjoint pair $(k_{\la}^v, k^v)$, one can easily check that $\FB$ and $\FC$ satisfy  Conditions $(P1)$, $(P2)$, $(P3)$, $(P4)$ and $(P5)$. So Theorem \ref{recollement} implies the following recollement
\vspace{0.3cm}
\[\xymatrix@C=0.5cm{\D(\Mod \FB) \ar[rrr]  &&& \D(\Mod \CS) \ar[rrr] \ar@/^1.5pc/[lll] \ar@/_1.5pc/[lll] &&& \D(\Mod \FC).\ar@/^1.5pc/[lll] \ar@/_1.5pc/[lll] }\]
\vspace{0.2cm}

On the other hand, there is an equivalence $(\CS_{\CQ_{I_*}}^B )^{\op}  \st{\sim}\lrt \FB $, that assigns every vertex $v$ to $k_{\la}^v(i_*(B))$. Indeed,
let $v$ and $w$ be vertices of $\CQ$. Then $$\begin{array}{lll}
\Hom_{\K^{\bb}(\Prj A\CQ/\SI)}(k_{\la}^v(i_*(B), k_{\la}^w(i_*(B)))) & \cong \Hom_{\K^{\bb}(\Prj A)}(i_*(B), \oplus_{\CQ_{I_*}(w,v)}i_*(B))\\
& \cong \oplus_{\CQ_{I_*}(w,v)}\End(i_*(B))\\
& \cong \oplus_{\CQ_{I_*}(w,v)} B. \end{array}$$
Hence the category $\FMod(\FB)$ is equivalent  to $\FMod((\CS_{\CQ_{I_*}}^B)^{\op})$ and so to $\Rep(\CQ_{I_*},B)$.  Similarly,  there is an equivalence between $\FMod (\FC)$ and $\Rep(\CQ_{I_*},C)$.
\end{proof}

\subsection{Recollements of Graded rings}
Methods that are used in this section can be applied to get recollements of derived categories of graded modules over graded rings. Recall that a ring $A$ with identity is called graded if there is a direct sum decomposition $A=\oplus_{i\in G}A_i$ (as additive subgroups) such that $A_iA_j\subseteq A_{ij}$, for all $i,j \in G$. A left graded $A$-module is a left $A$-module $M$ together with an internal direct sum decomposition $M=\oplus_{i\in G}M_i$, where $M_i$ is a subgroup of the additive group $M$ in which $A_iM_j\subseteq M_{ij}$ for all $i,j \in G$. We denote  the category of all graded left $A$-modules by $\gr A$. Moreover, let ${\rm gr}^{\geq 0} \mbox{-} A$, resp. ${\rm gr}^{<0} \mbox{-} A$ denote the category of positively, resp. negatively, graded $A$-modules, that is $M_i =0$ for $i<0$, resp. $i\geq 0$.

Let $A_{-\infty}^{+\infty}$,  $A_{-\infty}$ and $A^{+\infty}$ denote quivers
\[\begin{array}{lll}
\cdots \lrt -2 \lrt -1 \lrt 0 \lrt 1 \lrt 2 \lrt \cdots, \\  \cdots \lrt -3 \lrt -2 \lrt -1, \\ \text{and} \ \ \ \ 0 \lrt 1 \lrt 2 \lrt \cdots,
\end{array}\]
respectively.

Let $\CA$ be an abelian category. $\C^{\geq 0}(\CA)$, resp. $\C^{< 0}(\CA)$, denotes  the full subcategory of $\C(\CA)$ consisting of complexes $\X$ with $X^i=0$ for all $i <0$, resp. $i\geq 0$.

\begin{stheorem}\label{A-+}
For a ring  $B$, the following statements hold true.
\begin{itemize}
\item [$(i)$] The path ring $B A_{- \infty}^{+\infty}$ admits a $\D^-(\Mod)$ level recollement of the form
    \vspace{0.3cm}
\[\xymatrix@C=0.5cm{ \D^-(\Mod B A_{-\infty}) \ar[rrr]  &&& \D^-(\Mod BA_{-\infty}^{+\infty}) \ar[rrr] \ar@/^1.5pc/[lll] \ar@/_1.5pc/[lll] &&& \D^-(\Mod B A^{+\infty}).\ar@/^1.5pc/[lll] \ar@/_1.5pc/[lll] }\]
\vspace{0.3cm}
\item [$(ii)$] The category $\C(\Mod B)$ admits a $\D^-(\Mod)$ level recollement of the form
    \vspace{0.3cm}
\[\xymatrix@C=0.5cm{ \D^-(\C^{< 0}(\Mod B)) \ar[rrr]  &&& \D^-(\C(\Mod B)) \ar[rrr] \ar@/^1.5pc/[lll] \ar@/_1.5pc/[lll] &&& \D^-( \C^{\geq 0}(\Mod B)).\ar@/^1.5pc/[lll] \ar@/_1.5pc/[lll] }\]
\vspace{0.2cm}
\end{itemize}
\end{stheorem}

\begin{proof}
$(i)$ Consider $A_{-\infty}$ and $A^{+\infty}$ as subquivers of $A_{-\infty}^{+\infty}$. For every $v \in (A_{-\infty})_0$ and $w \in (A^{+\infty})_0$, set $\FB_v = e_{\la}^v(B)$ and $\FC_w= e_{\la}^w(B)$ in $\Mod B A_{-\infty}^{+\infty}$. The same argument as in the proof of Theorem \ref{Ext}, implies that $\FB_v$ and $\FC_w$ satisfy all conditions $(P1)$, $(P2)$, $(P3)$, $(P4)$ and $(P5)$. So by Theorem \ref{recollement}, we have   gets the desired  recollement
\vspace{0.3cm}
\[\xymatrix@C=0.5cm{ \D^-(\Mod B A_{-\infty}) \ar[rrr]  &&& \D^-(\Mod BA_{-\infty}^{+\infty}) \ar[rrr] \ar@/^1.5pc/[lll] \ar@/_1.5pc/[lll] &&& \D^-(\Mod B A^{+\infty}).\ar@/^1.5pc/[lll] \ar@/_1.5pc/[lll] }\]
\vspace{0.2cm}

$(ii)$ It can be easily seen that the category of complexes $\C(\Mod B)$ coincides with the category of all representations of the quiver
\[\xymatrix{ \cdots \ar[r] & i-1 \ar[r]^{\al^{i-1}} & i \ar[r]^{\al^i} & i+1 \ar[r]^{\al^{i+1}} &\cdots }\]
bound by the zero relations $\al^{i+1}\al^i=0$, for all $ i \in \Z$. Let, for every $i \in \Z^+$, resp. $j \in \Z^-$, $\FB_i$, resp. $\FC_j$, be the complex
\[ \cdots \lrt 0 \lrt B \lrt B \lrt 0 \lrt \cdots,\]
with $B$ on the left hand side sits on the $i$-th, resp. $j$-th, term.
It can be easily seen, using a simple modification of the proof of Theorem \ref{Ext}, that $\FB_i$ and $\FC_j$ satisfy all the required conditions to apply  Theorem \ref{recollement}. Hence there is the desired recollement.
\end{proof}

Observe that for a ring $B$, $\Mod B A_{-\infty}^{+\infty}$ is equivalent to $\gr B[x]$, when $B[x]$ is considered as a $\Z$-graded ring, with a copy of $B$, generated by $1$, in degree $0$ and a copy of $B$, generated by $x^n$, in degree $n$, for every $n \in \N$. Also, it is known that the category of complexes over $B$ is equivalent to $\gr B[x]/(x^2)$, where  $B[x]/(x^2)$ is viewed as a $\Z$-graded ring with a copy of $B$, generated by $1$, in degree $0$ and a copy of $B$, generated by $x$, in degree $1$ and zero elsewhere; see also \cite{GH}.

Therefore, in view of Theorem \ref{A-+}, we have the following recollements of graded modules over graded rings
\vspace{0.3cm}
\[\xymatrix@C=0.5cm{ \D^-({\rm gr}^{< 0} \mbox{-} B[x]) \ar[rrr]  &&& \D^-(\gr B[x]) \ar[rrr] \ar@/^1.5pc/[lll] \ar@/_1.5pc/[lll] &&& \D^-(  {\rm gr}^{\geq 0} \mbox{-} B[x])),\ar@/^1.5pc/[lll] \ar@/_1.5pc/[lll] }\]
\vspace{0.2cm}

\[\xymatrix@C=0.5cm{ \D^-({\rm gr}^{< 0} \mbox{-}B[x]/(x^2)) \ar[rrr]  &&& \D^-(\gr B[x]/(x^2)) \ar[rrr] \ar@/^1.5pc/[lll] \ar@/_1.5pc/[lll] &&& \D^-(  {\rm gr}^{\geq 0} \mbox{-} B[x]/(x^2))).\ar@/^1.5pc/[lll] \ar@/_1.5pc/[lll] }\]
\vspace{0.2cm}

The notion of $N$-complexes are introduced and studied in \cite{Es}. A similar argument as above can be applied to see that there exists an equivalence between the category of $N$-complexes and $\gr B[x]/(x^n)$. So as above, we can show that for every positive integer $n$, $\D^-(\gr \frac{B[x]}{(x^n)})$ admits a recollement
\vspace{0.3cm}
\[\xymatrix@C=0.5cm{ \D^-({\rm gr}^{< 0} \mbox{-}B[x]/(x^n)) \ar[rrr]  &&& \D^-(\gr B[x]/(x^n)) \ar[rrr] \ar@/^1.5pc/[lll] \ar@/_1.5pc/[lll] &&& \D^-(  {\rm gr}^{\geq 0} \mbox{-} B[x]/(x^n)))\ar@/^1.5pc/[lll] \ar@/_1.5pc/[lll] }\]
\vspace{0.3cm}
\\relative to $\D^-({\rm gr}^{< 0} \mbox{-}B[x]/(x^n))$ and $\D^-(  {\rm gr}^{\geq 0} \mbox{-} B[x]/(x^n)))$.

\end{document}